\documentclass[11pt,a4paper]{amsart}

\usepackage[utf8]{inputenc}

\title{The $\dbar$-equation on a non-reduced analytic space}

\author{Mats Andersson \& Richard L\"ark\"ang}
\thanks{The authors were partially supported by grants from the Swedish Research Council.}
\address{Department of Mathematical Sciences, Division of Algebra and Geometry,
Chalmers University of Technology and the  University of Gothenburg,
SE-412 96 G\"{o}teborg, Sweden}
\email{matsa@chalmers.se, larkang@chalmers.se}

\usepackage{amsmath}
\usepackage{amsthm,amssymb,latexsym}
\usepackage{url}
\usepackage{mathrsfs}
\usepackage{graphicx}
\usepackage{geometry}
\newtheorem{thm}{Theorem}[section]
\newtheorem{lma}[thm]{Lemma}
\newtheorem{cor}[thm]{Corollary}
\newtheorem{prop}[thm]{Proposition}

\theoremstyle{definition}

\newtheorem{df}[thm]{Definition}

\theoremstyle{remark}

\newtheorem{preremark}[thm]{Remark}
\newtheorem{preex}[thm]{Example}

\newenvironment{remark}{\begin{preremark}}{\qed\end{preremark}}
\newenvironment{ex}{\begin{preex}}{\qed\end{preex}}

\newcommand{\ta}{{\tau}}
\newcommand{\re}{{\text{Re}\,}}
\newcommand{\Image}{{\text{Im}\,}}
\newcommand{\Ker}{{\text{Ker}\,}}
\newcommand{\C}{\mathbb{C}}
\newcommand{\debar}{\bar{\partial}}
\newcommand{\dbar}{\bar{\partial}}
\newcommand{\A}{\mathscr{A}}

\newcommand{\J}{\mathcal{J}}
\newcommand{\E}{\mathscr{E}}

\newcommand{\W}{\mathcal{W}}
\newcommand{\PM}{\mathcal{PM}}

\newcommand{\hol}{\mathscr{O}}
\newcommand{\Ok}{\mathscr{O}}

\newcommand{\F}{\mathcal{F}}

\newcommand{\Dom}{\text{Dom}\,}

\newcommand{\CH}{\mathcal{{ CH}}}

\newcommand{\M}{\mathcal{M}}
\newcommand{\pmm}{pseudomeromorphic }
\newcommand{\nbh}{neighborhood }
\newcommand{\1}{{\bf 1}}
\newcommand{\w}{\wedge}
\newcommand{\codim}{{\text{codim}\,}}
\newcommand{\Ba}{{\scalebox{1.3}{$\omega$}}}
\newcommand{\Homs}{{\mathcal Hom}}

\newcommand{\Ho}{{\mathcal H}}
\newcommand{\Exts}{{\mathcal Ext\,}}
\newcommand{\Kers}{{\mathcal Ker\,}}
\newcommand{\ann}{{\rm ann\,}}

\newcommand{\U}{{\mathcal U}}
\newcommand{\Cu}{{\mathcal C}}

\newcommand{\Pk}{{\mathbb P}}
\DeclareMathOperator{\depth}{depth}
\DeclareMathOperator{\pd}{pd}

\DeclareMathOperator{\supp}{supp}

\date{\today}

\numberwithin{equation}{section}

\begin{document}

\nocite{*}
\bibliographystyle{plain}



\begin{abstract}
Let $X$ be a,  possibly non-reduced,
analytic space of pure dimension. We introduce a notion of $\overline{\partial}$-equation
on $X$ and prove a Dolbeault-Grothendieck lemma. We obtain
fine sheaves $\mathcal{A}_X^q$ of $(0,q)$-currents, so that the associated
Dolbeault complex yields a resolution
of the structure sheaf $\mathscr{O}_X$. Our construction is based on intrinsic
semi-global Koppelman formulas on $X$.
\end{abstract}

\maketitle
\thispagestyle{empty}

\section{Introduction}

Let $X$ be a smooth complex manifold of dimension $n$ and let $\E_X^{0,*}$ denote the sheaf of
smooth $(0,*)$-forms. It is well-known that the Dolbeault complex
\begin{equation}\label{dolb}
0\to \Ok_X\stackrel{i}{\to} \E_X^{0,0}\stackrel{\dbar}{\to} \E_X^{0,1}\stackrel{\dbar}{\to} \cdots \stackrel{\dbar}{\to} \E_X^{0,n}\to 0
\end{equation}
is exact, and hence provides a fine resolution of the structure sheaf $\Ok_X$.  If $X$ is a
reduced analytic space of pure dimension,
then there is still a natural notion of "smooth forms". In fact, assume that $X$
is locally embedded as
$i\colon X\to \Omega$, where $\Omega$ is a pseudoconvex domain in $\C^N$. If
$\Kers i^*$ denotes the subsheaf of all smooth forms $\xi$  in ambient space such that $i^*\xi=0$
on the regular part $X_{reg}$ of $X$,
then one defines the sheaf $\E_X$ of smooth forms on $X$ simply as
$$
\E_X:=\E_\Omega/\Kers i^*.
$$
It is well-known that  this definition is independent of the choice of embedding of $X$.
Currents on $X$ are defined as the duals of smooth forms with compact support.
It is readily seen that the currents $\mu$ on $X$ so defined are in a one-to-one correspondence to the
currents $\hat\mu=i_*\mu$ in ambient space such that $\hat\mu$ vanish on $\Kers i_*$, see, e.g.,
\cite{AS}.  There is an induced $\dbar$-operator on smooth forms
and currents on $X$.  In particular, \eqref{dolb} is a complex on $X$ but in  general it is not exact.
In \cite{AS},   Samuelsson and the first author introduced, by means of intrinsic Koppelman
formulas on $X$, fine sheaves $\A_X^*$ of $(0,*)$-currents that are smooth on $X_{reg}$ and with mild
singularities at the singular part of $X$, such that
\begin{equation}\label{dolb1}
0\to \Ok_X\stackrel{i}{\to} \A_X^{0}\stackrel{\dbar}{\to} \A_X^{1}\stackrel{\dbar}{\to} \cdots \stackrel{\dbar}{\to} \A_X^{n}\to 0
\end{equation}
is exact, and thus a fine resolution of the structure sheaf $\Ok_X$.  An immediate consequence
is the representation
\begin{equation}\label{kalla}
H^{q}(X,\Ok_X)=\frac{\Ker \big(\A^{0,q}(X)\stackrel{\dbar}{\to}\A^{0,q+1}(X)\big)}
{\text{Im}\big(\A^{0,q-1}(X)\stackrel{\dbar}{\to}\A^{0,q}(X)\big)}, \quad  q\ge 1,
\end{equation}
of sheaf cohomology,
and so  \eqref{kalla} is a generalization of the classical Dolbeault isomorphism.
In special cases more qualitative information of the sheaves $\A_X^q$ are known,
see, e.g., \cite{LR2,ALRSW}.

\smallskip
Starting with the influential works \cite{PaSt1,PaSt2} by Pardon and Stern, there has been a lot
of progress recently on the $L^2$-$\dbar$ theory on non-smooth (reduced) varieties; see, e.g.,
\cite{FOV,OvVass2,Rupp2}.
The point in these works, contrary to \cite{AS}, is basically to determine the  obstructions to solve
$\dbar$ locally in $L^2$.
For a more extensive list of references regarding the $\dbar$-equation on reduced singular varieties, see, e.g., \cite{AS}.

In \cite{HePo2}, a notion of the $\dbar$-equation on non-reduced local complete intersections
was introduced, and which was further studied in \cite{HePo3}. We discuss below how their work relates to ours.

\smallskip

The aim of this paper is to extend 
the construction in \cite{AS} to a non-reduced
pure-dimensional analytic space.  The first basic problem is to find appropriate
definitions of forms and currents on $X$.
Let $X_{reg}$ be the part of $X$ where the
underlying reduced space $Z$ is smooth, and in addition $\Ok_X$
is Cohen-Macaulay.  On $X_{reg}$ the structure sheaf  $\Ok_X$ has a structure
as a free finitely generated $\Ok_Z$-module.
More precisely, assume that
 we have a local embedding $i\colon X\to \Omega\subset\C^N$ and coordinates $(z,w)$ in $\Omega$ such that
$Z=\{w=0\}$. Let  $\J$ be the defining ideal sheaf for $X$ on $\Omega$.
Then there are monomials $1, w^{\alpha_1}, \ldots, w^{\alpha_{\nu-1}}$
such that each $\phi$ in $\Ok_\Omega/\J\simeq\Ok_X$ has a unique representation
\begin{equation}\label{burk}
\phi=\hat\phi_0\otimes 1+\hat\phi_1\otimes w^{\alpha_1} +\cdots + \hat\phi_{\nu-1}\otimes w^{\alpha_{\nu-1}},
\end{equation}
where $\hat\phi_j$ are in $\Ok_Z$.  A reasonable  notion of a smooth form on $X$
should admit a similar representation on $X_{reg}$ with smooth forms
$\hat\phi_j$ on $Z$.
We first introduce the sheaves $\E_X^{0,*}$ of smooth $(0,*)$-forms
on $X$.  By duality, we then obtain the sheaf $\Cu_X^{n,*}$ of $(n,*)$-currents.
We are mainly interested in the subsheaf $\PM_X^{n,*}$ of pseudomeromorphic currents, and especially,
the even more restricted sheaf $\W_X^{n,*}$ of such currents with the so-called
standard extension property, SEP, on $X$. A current with the SEP is, roughly speaking, determined
by its restriction to any dense Zariski-open subset.

Of special interest is the sheaf $\Ba_X^n\subset\W_X^{n,0}$  of $\dbar$-closed \pmm $(n,0)$-currents. In the
reduced case this is precisely the sheaf of holomorphic $(n,0)$-forms in the sense of
Barlet-Henkin-Passare, see, e.g., \cite{Barlet,HePa}.

We have no definition of "smooth $(n,*)$-form" on $X$.
In order to define $(0,*)$-currents,
we use instead the sheaf $\Ba_X^n$ in the following way. Any holomorphic function
defines a morphism in $\Homs(\Ba^n_X,\Ba^n_X)$, and it is a reformulation
of a fundamental result of Roos, \cite{Roos}, that this morphism is indeed injective,
and generically surjective. In the reduced case, multiplication by a current in $\W_X^{0,*}$
induces a morphism in $\Homs(\Ba^n_X,\W_X^{n,*})$, and in fact
$\W_X^{0,*}\to \Homs(\Ba^n_X,\W_X^{n,*})$ is an isomorphism.
In the non-reduced case, we then take this as the
definition of $\W_X^{0,*}$. It turns out that with this definition, on $X_{reg}$,
any element of $\W_X^{0,*}$ admits a unique representation \eqref{burk}, where
$\hat\phi_j$ are in $\W_Z^{0,*}$, see Section \ref{prut} below for details.

Given $v,\phi$ in $\W_X^{0,*}$ we say that $\dbar v=\phi$ if $\dbar (v \w h)=\phi\w h$ for all $h$ in $\Ba_X^n$.
Following  \cite{AS} we introduce semi-global integral formulas and prove that if $\phi$ is a smooth
$\dbar$-closed $(0,q+1)$-form there is locally a current $v$ in $\W_X^{0,q}$ such that
$\dbar v=\phi$. A
crucial problem is to verify that the integral operators preserve smoothness on $X_{reg}$
so that the solution $v$ is indeed smooth on $X_{reg}$.  By an iteration procedure as in \cite{AS} we can
define sheaves $\A_X^k\subset\W_X^{0,k}$ and obtain our main result in this paper.

\begin{thm} \label{main}
Let $X$ be an analytic space of pure dimension $n$.
There are sheaves $\A_X^k\subset\W_X^{0,k}$ that are modules over $\E_X^{0,*}$, coinciding with $\E_X^{0,k}$ on
$X_{reg}$,  and such that \eqref{dolb1} is a resolution of the structure sheaf $\Ok_X$.
\end{thm}

The main contribution in this article compared to \cite{AS} is the development of a theory for smooth $(0,*)$-forms and various classes of $(n,*)$- and $(0,*)$-currents
in the non-reduced case as is described above. This is done in Sections \ref{boxer}-\ref{sect:dbarW0}. The construction of integral operators to provide solutions to $\dbar$
in Section~\ref{sect:solveDbar} and the construction of the fine resolution of $\Ok_X$ in Section \ref{fine}, which proves Theorem~\ref{main},
are done pretty much in the same way as in \cite{AS}. The proof of the smoothness of the
solutions of the regular part  in Section~\ref{pudding} however becomes significantly more involved in the non-reduced case and requires completely new ideas.
In Section~\ref{fullst} we discuss the relation to the results in \cite{HePo2,HePo3}
in case $X$ is a local complete intersection.

\subsection*{Acknowledgements}

We thank the referee for very careful reading and many valuable remarks.

\section{Pseudomeromorphic currents} \label{sect:pm}

Let $s_1,\ldots, s_m$ be coordinates in $\C^m$,  let $\alpha$ be a smooth form with compact support, and
let $a_1,\ldots, a_r$ be positive integers, $0\le\ell\le r\le m$. Then
$$
\dbar\frac{1}{s_1^{a_1}}\w \cdots \w\dbar\frac{1}{s_\ell^{a_\ell}}\w
\frac{\alpha}{s_{\ell+1}^{a_{\ell+1}}\cdots s_r^{a_r}}
$$
is a well-defined current that we call an {\it elementary (pseudomeromorphic) current}.
Let $Z$ be a reduced space of pure dimension.
A current $\tau$ is {\it pseudomeromorphic} on $Z$ if, locally, it  is the push-forward of a finite sum of
elementary \pmm currents under a sequence of modifications,
simple projections, and open inclusions.  The \pmm currents define an analytic sheaf $\PM_Z$
on $Z$.  This sheaf was introduced in \cite{AW2} and somewhat extended in \cite{AS}. If nothing else
is explicitly stated, proofs of the properties listed below can be found in, e.g.,  \cite{AS}.

If $\tau$ is \pmm and has support on an analytic subset $V$, and  $h$ is a holomorphic function that
vanishes on $V$, then $\bar h \tau=0$ and $d\bar h\w \tau=0$.

Given a \pmm current $\tau$ and a subvariety $V$ of some open subset
$\U\subset Z$, the natural restriction to the open set $\U\setminus V$
of $\tau$ has a natural extension to a \pmm current on $\U$ that we denote by $\1_{\U\setminus V}\tau$.
Throughout this paper we let $\chi$ denote a smooth function on $[0,\infty)$ that is
$0$ in a \nbh of $0$ and $1$ in a \nbh of $\infty$.
If $h$ is a holomorphic tuple whose common zero set is $V$, then
\begin{equation} \label{restrDef}
\1_{\U\setminus V}\tau=\lim_{\epsilon\to 0^+}\chi(|h|^2/\epsilon)\tau.
\end{equation}
Notice that
$
\1_V\tau:=(1-\1_{\U\setminus V})\tau
$
is also pseudomeromorphic and has support on $V$.
If $W$ is another analytic set, then
\begin{equation}\label{stare}
\1_V\1_W\tau=\1_{V\cap W}\tau.
\end{equation}
This action of $\1_V$ on the sheaf of \pmm currents is a basic tool.
In fact one can extend this calculus to all constructible sets so that \eqref{stare} holds, see \cite{AW2}.
One readily checks that if $\xi$ is a smooth form, then
\begin{equation}\label{stare2}
\1_V(\xi\w\tau)=\xi\w \1_V\tau.
\end{equation}
If $f\colon Z'\to Z$ is a modification  and $\tau$ is in $\PM_{Z'}$ then $f_*\tau$ is in $\PM_Z$.
The same holds if $f$ is a simple projection and $\tau$ has compact support in the
fiber direction. In any case we have
\begin{equation}\label{stare3}
\1_V f_*\tau=f_*(\1_{f^{-1}V}\tau).
\end{equation}
It is not hard to check that if $\tau$ is in $\PM_Z$ and $\tau'$ is in $\PM_{Z'}$, then
$\tau\otimes\tau'$ is in $\PM_{Z\times Z'}$, see, e.g., \cite[Lemma 3.3]{Apm}. If $V\subset \U\subset Z$ and $V'\subset\U'\subset Z'$,  then
\begin{equation}\label{bost}
(\1_V\tau)\otimes\1_{V'}\tau'=\1_{V\times V'}(\tau\otimes\tau').
\end{equation}


Another basic tool is the {\it dimension principle},
that states that if $\tau$ is a \pmm $(*,p)$-current with support on
an analytic set with codimension larger than $p$, then $\tau$ must vanish.

A \pmm current $\tau$ on $Z$ has the {\it standard extension property}, SEP, if
$\1_V\tau=0$ for each germ $V$ of an analytic set with positive codimension on $Z$. The set $\W_Z$ of all \pmm
currents on $Z$ with the SEP is a subsheaf of $\PM_Z$.
By \eqref{stare2}, $\W_Z$ is closed under multiplication by smooth forms.

\smallskip
Let $f$ be a holomorphic function (or a holomorphic section of a Hermitian line bundle),
not vanishing identically on any irreducible component of $Z$.
Then $1/f$, a priori defined outside of $\{ f = 0 \}$, has an extension
as a pseudomeromorphic current, the principal value current, still denoted by $1/f$,
such that $\1_{\{ f = 0 \}}( 1/f) = 0$. The current $1/f$ has the SEP and
\begin{equation*} 
    \frac{1}{f} = \lim_{\epsilon \to 0^+} \chi(|f|^{2}/\epsilon) \frac{1}{f}.
\end{equation*}
%
%
We say that a current $a$ on $Z$ is {\it almost semi-meromorphic} if there
is a modification $\pi\colon Z'\to Z$, a holomorphic section $f$ of a line bundle
$L\to Z'$ and a smooth form $\gamma$ with values in $L$ such that
$a=\pi_*(\gamma/f)$, cf.\ \cite[Section~4]{AW3}.
If $a$ is almost semi-meromorphic, then it is clearly pseudomeromorphic. Moreover,
it is smooth outside
an analytic set $V\subset Z$ of positive codimension, $a$ is in $\W_Z$, and in particular,
$a=\lim_{\epsilon\to 0^+}\chi(|h|/\epsilon)a$ if $h$ is a holomorphic tuple that cuts out
(an analytic set of positive codimension that contains) $V$.
The \emph{Zariski singular support} of $a$ is the Zariski closure of the set where $a$ is not smooth.

One can multiply \pmm currents by almost semi-meromorphic currents;
and this fact will be crucial in defining $\W_X^{0,*}$, when $X$ is non-reduced.
Notice that if $a$ is almost semi-meromorphic in $Z$ then it also is in any
open $\U\subset Z$.

\begin{prop}[{\cite[Theorem~4.8, Proposition 4.9]{AW3}}]\label{ball}
Let $Z$ be a reduced space, assume that $a$ is an almost semi-meromorphic current in $Z$, and let
$V$ be the Zariski singular support of $a$.

\smallskip
\noindent (i) If $\tau$ is a pseudomeromorphic current in $\U\subset Z$,
then there is a unique \pmm current
$a\w\tau$ in $\U$ that coincides with (the naturally defined current)
$a\w\tau$ in $\U\setminus V$ and such that $\1_V (a\w\tau)=0$.

\smallskip
\noindent (ii) If $W\subset \U$ is any analytic subset, then
\begin{equation}\label{krut}
\1_W (a\w\tau)=a\w\1_W\tau.
\end{equation}
\end{prop}

Notice  that if $h$ is a tuple that cuts out $V$, then in view of \eqref{restrDef},
\begin{equation} \label{proddef}
a\w\tau=    \lim_{\epsilon \to 0^+} \chi(|h|^2/\epsilon) a \wedge \tau.
\end{equation}
It follows that  if $\xi$ is a smooth form, then
\begin{equation}\label{korp}
\xi\w (a\w \tau)=(-1)^{\deg \xi \deg a}a\w (\xi\w \tau).
\end{equation}

For future reference we will need the following result.

\begin{prop}\label{trilsk}
Let $Z$ be a reduced space. Then $\PM_Z=\W_Z+\dbar \W_Z$.
\end{prop}

\begin{proof}  First assume that $Z$ is smooth. Since $\W_Z$ is closed under multiplication by
smooth forms, so is $\W_Z + \dbar \W_Z$. The statement that $\PM_Z = \W_Z + \dbar \W_Z$ is local, and since both sides are closed
under multiplication by cutoff functions, we may consider
a \pmm current $\mu$ with compact support in $\C^n$. If $\mu$ has bidegree $(*,0)$, then
it is in $\W_Z$ in view of the dimension principle. Thus we assume that $\mu$ has bidegree
$(*,q)$ with $q\ge 1$.
Let
\begin{equation} \label{koppelman-current}
K\mu(z)=\int_\zeta k(\zeta,z) \w\mu(\zeta),
\end{equation}
where $k$ is the Bochner-Martinelli kernel. Here \eqref{koppelman-current} means
that $K\mu=p_*(k\w \mu\otimes 1)$, where $p$ is the projection
$
\C^n_\zeta\times\C^n_z\to\C^n_z,\quad (\zeta,z)\mapsto z.
$
Recall that we have the Koppelman formula
$
\mu=\dbar K\mu +K(\dbar\mu)
$.
It is thus enough to see that $K\mu$ is in  $\W_Z$ if $\mu$ is pseudomeromorphic.
Let $\chi_\epsilon=\chi(|\zeta-z|^2/\epsilon)$.  It is easy to see, by a blowup of $\C^n\times\C^n$ along the diagonal,
that $k$ is almost semi-meromorphic on $\C^n\times\C^n$.   Thus, by \eqref{proddef},
$\chi_\epsilon k\w (\mu\otimes 1)\to k\w (\mu\otimes 1)$.
In view of Proposition~\ref{ball} it follows that $k\w( \mu\otimes 1)$ is pseudomeromorphic.
Finally, if $W$ is a germ of a subvariety of $\C^n$ of positive codimension, then by \eqref{stare3} and \eqref{bost},
\begin{multline*}
\1_W p_*(k\w \mu\otimes 1)=
\lim_{\epsilon \to 0^+}p_*(\1_{\C^n\times W}(\chi_\epsilon k\w (\mu\otimes 1)))=\\
\lim_{\epsilon \to 0^+}p_*(\chi_\epsilon k\w (\1_{\C^n\times W}\mu\otimes 1))
=\lim_{\epsilon \to 0^+} p_*(\chi_\epsilon k\w (\1_{\C^n}\mu\otimes \1_W 1))=0,
\end{multline*}
since $\1_W 1=0$.
Thus $K\mu$ is in $\W_Z$.

If $Z$ is not smooth, then we take a smooth modification $\pi\colon Z'\to Z$. For any
$\mu$ in $\PM_Z$ there is some $\mu'$ in $\PM_{Z'}$ such that $\pi_*\mu'=\mu$, see \cite[Proposition~1.2]{Apm}.
Since $\mu'=\tau+\dbar u$ with $\tau,u$ in $\W_{Z'}$, we have that
$\mu=\pi_*\tau+\dbar\pi_* u$.
\end{proof}

\subsection{Pseudomeromorphic currents with support on a subvariety}\label{struts}
Let $\Omega$ be an open set in $\C^N$ and let $Z$ be a (reduced) subvariety of pure
dimension $n$.   Let $\PM^Z_\Omega$ denote the sheaf of pseudomeromorphic
currents $\tau$ on $\Omega$ with support on $Z$, and let $\W_\Omega^Z$ denote the subsheaf of
$\PM^Z_\Omega$ of currents of bidegree $(N,*)$ with the SEP with respect to $Z$,
i.e., such that  $\1_W \tau=0$ for all germs $W$ of subvarieties of $Z$ of positive codimension.
The sheaf
$\CH_\Omega^Z$ of Coleff-Herrera currents on $Z$ is the subsheaf of $\W_\Omega^Z$ of
$\dbar$-closed $(N,p)$-currents, where $p=N-n$.

\begin{remark}
In  \cite{AS, Aext}   $\CH_Z^\Omega$  denotes the sheaf of pseudomeromorphic $(0,p)$-currents
with support on $Z$ and the SEP with respect to $Z$. If this sheaf is tensored
by the canonical bundle $K_\Omega$ we get the sheaf $\CH^Z_\Omega$ in this paper.
Locally these sheaves are thus isomorphic via the mapping
$\mu\mapsto\mu\w \alpha$, where $\alpha$ is a
non-vanishing holomorphic $(N,0)$-form.
\end{remark}

We have the following direct consequence of Proposition~\ref{ball}.

\begin{prop} \label{asmW}
Let $Z \subset\Omega$ be a subvariety of pure dimension,  let $a$ be almost semi-meromorphic
in $\Omega$, and assume that it is smooth generically on $Z$.
If $\tau$ is in $\W_\Omega^Z$, then $a\w\tau$ is in $\W_\Omega^Z$ as well.
\end{prop}

Assume that we have local coordinates $(z,w) \in \C^n \times \C^p$ in $\Omega$ such that $Z=\{w=0\}$. We will use the
short-hand notation
$$
\dbar\frac{dw}{w^{\gamma+\1}}:=\dbar\frac{dw_1}{w_1^{\gamma_1+1}}\w\cdots\w\dbar
\frac{dw_p}{w_p^{\gamma_p+1}}
$$
for multiindices  $\gamma=(\gamma_1,\ldots,\gamma_p)$ with $\gamma_j\ge 0$,
and let $\gamma! := \gamma_1! \cdots \gamma_p!$.
Notice that that
\begin{equation}\label{spurt}
\frac{1}{(2\pi i)^p}\dbar\frac{dw}{w^{\gamma+\1}} . \xi =\frac{1}{\gamma !} \int_z \frac{\partial^\gamma\xi}{\partial w^\gamma}(z,0)
\end{equation}
for test forms $\xi$.
If $\tau$ is in $\W_Z$, then it follows by \eqref{bost} and the fact that $\supp \dbar(1/w^{\gamma+\1}) = \{ w = 0 \}$ that
$\tau\otimes\dbar(1/w^{\gamma+\1})$ is in $\W_\Omega^Z$.
We have the following local structure result, see \cite[Proposition~4.1 and (4.3)]{AW4}
and \cite[Theorem~3.5]{AW3}.

\begin{prop}\label{parasol}
Assume that we have local coordinates $(z,w)$ such that $Z=\{w=0\}$. Then $\tau$ in $\W_\Omega^Z$ has
a unique representation as a finite sum
\begin{equation}\label{paraply}
\tau=\sum_\gamma \tau_\gamma \wedge dz \otimes\dbar\frac{dw}{w^{\gamma+\1}},\quad \tau_\gamma\in\W_Z^{0,*},
\end{equation}
where $dz := dz_1 \w \cdots \w dz_n$.
If $\pi$ is the projection $(z,w)\mapsto z$, then
\begin{equation}\label{parsol1}
\tau_\gamma \wedge dz=(2\pi i)^{-p} \pi_*(w^\gamma \tau).
\end{equation}
\end{prop}

If in addition $\dbar\tau$ is in $\W_\Omega^Z$  then its coefficients in the expansion \eqref{paraply} are $\dbar\tau_\gamma$, cf.\ \eqref{parsol1}.
In particular,  $\dbar\tau=0$ if and only if $\dbar\tau_\gamma=0$ for all $\gamma$.

\smallskip

Let us now consider the pairing between  $\W_\Omega^Z$ and germs $\phi$ at $Z$ of smooth
$(0,*)$-forms. We assume that $Z$ is smooth and that we have coordinates $(z,w)$ as before, that  $\tau$ is in $\W_\Omega^Z$, and that \eqref{paraply} holds.
Moreover, we assume that
 $\phi$ is a smooth $(0,*)$-form in a \nbh of $Z$ in $\Omega$. For any positive integer
$M$ we have the expansion
\begin{equation}\label{paraply3}
\phi=\sum_{|\alpha| < M} \phi_\alpha(z)\otimes w^\alpha+\Ok(|w|^M)+\Ok(\bar w,d\bar w),
\end{equation}
where
$$
\phi_\alpha(z)=\frac{1}{\alpha!}\frac{\partial\phi}{\partial w^\alpha}(z,0)
$$
and $\Ok(\bar w,d\bar w)$ denotes a sum of terms, each of which contains a factor $\bar w_j$ or $d\bar w_j$
for some $j$.
If $M$ in \eqref{paraply3} is chosen so that $\Ok(|w|^M)\tau=0$, then
$$
\phi\w \tau=\sum_{\alpha\le \gamma}\phi_\alpha\w \tau_\gamma \wedge dz \otimes\dbar\frac{dw}{w^{\gamma-\alpha+\1}},
$$
i.e.,
\begin{equation}\label{paraply2}
\phi\w \tau=\sum_{\ell\ge 0}\sum_{\gamma\ge 0} \phi_{\gamma}\w \tau_{\ell+\gamma} \wedge dz \otimes\dbar\frac{dw}{w^{\ell+\1}}.
\end{equation}
Thus
$\phi\w \tau=0$ if and only if
$
\sum_{\gamma\ge 0} \phi_{\gamma}\w\tau_{\ell+\gamma}  =0
$
for all $\ell$ (which is a finite number of conditions!).

\subsection{Intrinsic pseudomeromorphic currents on a reduced subvariety}

Currents on a reduced analytic space $Z$ are defined as the dual of the sheaf of
test forms. If $i : Z \to Y$ is an embedding of a reduced space $Z$ into a smooth
manifold $Y$, then the push-forward mapping $\tau \mapsto i_* \tau$ gives an isomorphism
between currents $\tau$ on $Z$ and currents $\mu$ on $Y$ such that $\xi \wedge \mu = 0$
for all $\xi$ in $\E_Y$ such that $i^* \xi = 0$.

When defining pseudomeromorphic currents in the non-reduced case it is
desirable that it coincides with the previous definition in case $Z$ is reduced.
From \cite[Theorem~1.1]{Apm} we have the following description of pseudomeromophicity
from the point of view of an ambient smooth space.

\begin{prop}\label{L1}
 Assume that we have an embedding
$i\colon Z\to Y$ of a reduced space $Z$ into a smooth manifold $Y$.

\smallskip
\noindent (i) \  If $\tau$ is in $\PM_Z$, then $i_*\tau$ is in $\PM_Y$.

\smallskip
\noindent (ii) \ If $\tau$ is a current on $Z$ such that $i_* \tau$ is in $\PM_Y$ and $\1_{Z_{sing}}(i_* \tau)=0$,
then $\tau$ is in $\PM_Z$.
\end{prop}

Since $i_*( i^* \chi(|h|^2/\epsilon) \tau) = \chi(|h|^2/\epsilon) i_* \tau$ for any current $\tau$ on $Z$,
we get by \eqref{restrDef} that for a subvariety $V \subset \U\subset  Z$,
\begin{equation} \label{stare2prim}
    \1_V (i_* \tau) = i_*( \1_V \tau),
\end{equation}
i.e., \eqref{stare3} holds also for an embedding $i : Z \to Y$.
The condition $\1_{Z_{sing}}(i_* \tau) = 0$ in (ii) is fulfilled if $i_* \tau$ has the SEP with respect to $Z$.

\begin{cor} \label{Wextrinsic}
    We have the isomorphism
    \begin{equation*}
        i_* : \W^{n,*}_Z \to \Homs(\Ok_\Omega/\J,\W^Z_\Omega),
    \end{equation*}
    where $\J$ is the ideal defining $Z$ in $\Omega$.
\end{cor}

Notice that $\Homs(\Ok_\Omega/\J,\W_\Omega^Z)$ is precisely the sheaf of $\mu$ in $\W_\Omega^Z$
such that $\J\mu=0$.

\begin{proof}
    The map $i_*$ is injective, since it is injective on any currents, and it maps into
    $\Homs(\Ok_\Omega/\J,\W^Z_\Omega)$ by \eqref{stare2prim}.
    To see that $i_*$ is surjective, we take a $\mu$ in $\Homs(\Ok_\Omega/\J,\W^Z_\Omega)$.
    We assume first that we are on $Z_{\rm reg}$, with local coordinates such that $Z_{\rm reg} = \{ w = 0 \}$.
    If $\xi$ is in $\E_\Omega^{0,*}$ and $i^*\xi=0$, then $\xi$ is a sum of forms with a factor $d \bar w_j$,
    $w_j$ or $\bar w_j$. Since $w_j \in \J$, $w_j$ annihilates $\mu$ by assumption, and since
    $w_j$ vanishes on the support of $\mu$, $\bar w_j$ and $d \bar w_j$ annihilate $\mu$ since $\mu$
    is pseudomeromorphic. Thus, $\mu . \xi = 0$, so $\mu = i_* \tau$ for some current $\tau$ on $Z$.
    By Proposition~\ref{L1} (ii), $\tau$ is pseudomeromorphic, and by \eqref{stare2prim},
    has the SEP, i.e., $\tau$ is in $\W^{n,*}_Z$.
\end{proof}


\begin{remark}\label{soda}
We do not know whether $i_* \tau \in\PM_\Omega^Z$ implies that
$\tau\in\PM_Z$.
\end{remark}

By \cite[Proposition~3.12 and Theorem~3.14]{AW4}, we get

\begin{prop} \label{Wsmooth}
    Let $\varphi$ and $\phi_1,\dots,\phi_m$ be currents in $\W_Z$.
    If $\varphi = 0$ on the set on $Z_{reg}$ where $\phi_1,\dots,\phi_m$ are smooth,
    then $\varphi = 0$.
\end{prop}

\section{Local embeddings of a non-reduced analytic space}\label{get}
Let $X$ be an analytic space of pure dimension $n$ with structure sheaf $\Ok_X$
and let $Z=X_{red}$ be the underlying reduced analytic space.
For any point $x\in X$ there is, by definition,
an open set $\Omega\subset\C^N$ and an
ideal sheaf $\J\subset \Ok_\Omega$ of pure dimension $n$ with zero set $Z$
such that $\Ok_X$ is isomorphic to $\Ok_\Omega/\J$, and all associated primes of $\J$ at any point have dimension $n$.
We say that we have a local embedding $i\colon X\to \Omega\subset\C^N$ at $x$.
There is a minimal such $N$, called the Zariski embedding dimension $\hat N$ of $X$ at $x$,
and the associated embedding is said to be minimal.  Any two minimal
embeddings are identical up to a biholomorphism, and  any embedding
$i\colon X\to \Omega$ has locally at $x$ the form
\begin{equation}\label{embedding}
X\stackrel{j}{\to}\widehat\Omega\stackrel{\iota}{\to} \Omega:=\widehat\Omega\times\U, \quad i=\iota\circ j,
\end{equation}
where $j$ is minimal, $\U$ is an open subset of $\C^{m}_w$, $m=N-\hat N$,
and the ideal in  $\Omega$ is $\J=\widehat\J\otimes 1 +(w_1,\ldots,w_m)$.
Notice that we then also have embeddings $Z\to\widehat\Omega\to\Omega$;
however, the first one is in general not minimal.


Now consider a fixed local embedding  $i\colon X\to\Omega \subset \C^N$, assume that $Z$ is smooth,
and let $(z,w)$ be coordinates in  $\Omega$ such that $Z=\{w=0\}$. We can identify
$\Ok_Z$ with holomorphic functions of $z$, and we can define an injection
$$
\Ok_Z\to\Ok_X,\quad \phi(z)\mapsto \tilde\phi(z,w)=\phi(z).
$$
In this way $\Ok_X$ becomes an $\Ok_Z$-module, which however 
depends on the choice of coordinates.

\begin{prop}\label{sorelle}
Assume that $Z$ is smooth. Let $\Ok_X$ have the $\Ok_Z$-module structure from a choice of local coordinates
as above. Then $\Ok_X$ is a coherent $\Ok_Z$-module, and $\Ok_X$ is a free $\Ok_Z$-module at $x$ if and only if
$\Ok_X$ is Cohen-Macaulay at $x$.
\end{prop}

Recall that $f_1,\dots,f_m  \in R$ is a \emph{regular sequence} on the $R$-module $M$ if
$f_i$ is a non zero-divisor on $M/(f_1,\dots,f_{i-1})$ for $i=1,\dots,m$, and $(f_1,\dots,f_m) M \neq M$.
If $R$ is a local ring, then $\depth_RM$ is the maximal length $d$ of a regular sequence $f_1,\dots,f_d$
such that $f_1,\dots,f_d$ are contained in the maximal ideal $\mathfrak{m}$; furthermore,
$M$ is \emph{Cohen-Macaulay} if $\depth_R M = \dim_R M$, where $\dim_R M = \dim_R (R/\ann_R M)$.
If $R$ is Cohen-Macaulay, and $M$ has a finite free resolution over $R$, then
the \emph{Auslander-Buchsbaum} formula, \cite[Theorem~19.9]{Eis}, gives that
\begin{equation}\label{aus}
    \depth_R M + \pd_R M = \dim_R R,
\end{equation}
where $\pd_R M$ is the length of a minimal free resolution of $M$ over $R$.
In this case, $M$ is Cohen-Macaulay as
an $R$-module if and only if $M$ has a free resolution over $R$
of length $\codim M$.

\begin{remark}\label{ch}
Notice that if we have a local embedding $i\colon X\to\Omega$ as above,  then the depth and dimension of
$\Ok_{X,x}=\Ok_{\Omega,x}/\J$ as an $\Ok_{\Omega,x}$-module coincide with the depth and dimension
of $\Ok_{X,x}$ as an $\Ok_{X,x}$-module. Thus $\Ok_{X,x}$ is Cohen-Macaulay as an $\Ok_{X,x}$-module if and only if it is
Cohen-Macaulay as an $\Ok_{\Omega,x}$-module, and this holds in turn if and only if $\Ok_{\Omega,x}/\J$ has a free
resolution of length $N-n$.
\end{remark}

\begin{proof}[Proof of Proposition~\ref{sorelle}]
By the Nullstellensatz there is an $M$ such that $w^\alpha$ is in $\J$
in some neighborhood of $x$ if $|\alpha|=M$.  Let $\M\subset\Ok_\Omega$ be
the ideal generated by $\{w^\alpha;\ |\alpha|=M\}$.
Then $\M'=\Ok_\Omega/\M$ is a free, finitely generated $\Ok_Z$-module.
Thus, $\Ok_\Omega/\J \simeq \M'/ \J \M'$ is a coherent $\Ok_Z$-module,
which we note is generated by the finite set of monomials $w^\alpha$ such that
$|\alpha| < M$.


We shall  now  show that
\begin{equation} \label{eqdepth}
    \depth_{\Ok_{X,x}} \Ok_{X,x} = \depth_{\Ok_{Z,x}} \Ok_{X,x}
\end{equation}
and
\begin{equation} \label{eqdimension}
    \dim_{\Ok_{X,x}} \Ok_{X,x} = \dim_{\Ok_{Z,x}} \Ok_{X,x}.
\end{equation}
We claim that  a sequence $f_1,\dots,f_m$ in $\Ok_{X,x}$ is regular (on $\Ok_{X,x}$) if and only if
$\tilde f_1,\dots,\tilde f_m\in\Ok_{Z,x}$ is regular on $\Ok_{X,x}$, where $\tilde f_j(z)=f_j(z,0)$.
In fact, since $\Ok_{X,x}$ has pure dimension, a function $g\in\Ok_{X,x}=\Ok_{\Omega,x}/\J$ is a non zero-divisor
if and only if
$g$ is generically non-vanishing on each irreducible component of $Z(\J)$. Thus $f_1$ is a non zero-divisor
if and only if $\tilde f_1$ is. If it is, then $\Ok_{X,x}/(f_1)=\Ok_{\Omega,x}/(\J+(f_1))$ again has pure dimension.
Thus the claim follows by induction, and the fact that $Z(\J+(f_1,\dots,f_k)) = Z(\J+(\tilde{f}_1,\dots,\tilde{f}_k))$.
The claim immediately implies \eqref{eqdepth}.

To see \eqref{eqdimension}, we note first that $\dim_{\Ok_{X,x}} \Ok_{X,x}$ is just the usual (geometric) dimension
of $X$ or $Z$, i.e., in this case, $n$. Now,  $\ann_{\Ok_{Z,x}} \Ok_{X,x} = \{ 0 \}$,
so $\dim_{\Ok_{Z,x}} \Ok_{X,x} = \dim_{\Ok_{Z,x}} \Ok_{Z,x}/(\ann_{\Ok_{Z,x}} \Ok_{X,x}) = \dim_{\Ok_{Z,x}} \Ok_{Z,x} = n$.

From \eqref{eqdepth} and \eqref{eqdimension} we conclude that
$\Ok_{X,x}$ is Cohen-Macaulay as an $\Ok_{Z,x}$-module if and only if it is Cohen-Macaulay (as an $\Ok_{X,x}$-module).
Hence, by \eqref{aus}, with $R = \Ok_{Z,x}$ and $M = \Ok_{X,x}$,
\begin{equation*}
    \depth_{\Ok_{Z,x}} \Ok_{X,x} + \pd_{\Ok_{Z,x}} \Ok_{X,x} = n,
\end{equation*}
so  $\Ok_{X,x}$ is Cohen-Macaulay as an $\Ok_{Z,x}$-module if and only if
$\pd_{\Ok_{Z,x}} \Ok_{X,x} = 0$, that is, if and only if  $\Ok_{X,x}$ is a free $\Ok_{Z,x}$-module.
\end{proof}

In the proof above, we saw that
$\Ok_X$ is generated (locally) as an $\Ok_Z$-module by all monomials $w^\alpha$ with $|\alpha|\le M$ for some $M$.

\begin{cor}\label{sorelle2}
Assume that  $1, w^{\alpha_1},\ldots,w^{\alpha_{\nu-1}}$ is a minimal set of generators at a given point $x$
(clearly $1$ must be among the generators!). Then we have a unique representation
\eqref{burk} for each $\phi\in\Ok_{X,x}$ if and only if $\Ok_{X,x}$ is Cohen-Macaulay.
\end{cor}

By  coherence it follows that if $\Ok_{X,x}$ is free as an $\Ok_{Z,x}$-module, then
$\Ok_{Z,x'}$ is free as an $\Ok_{Z,x'}$-module for all $x'$ in a \nbh
of $x$, and  $1, w^{\alpha_1},\ldots,w^{\alpha_{\nu-1}}$ is a basis at each such $x'$.

\begin{ex}\label{steira}
Let $\J$ be the ideal in $\C^4$ generated by
$
(w_1^2, w_2^2, w_1 w_2, w_1z_2-w_2 z_1).
$
It is readily  checked that $\Ok_X$ is a free $\Ok_Z$-module
at a point on $Z=\{w_1=w_2=0\}$ where $z_1$ or
$z_2$ is $\neq 0$.  If, say, $z_1\neq 0$, then we can take
$1, w_1$ as generators.  At the point $z=(0,0)$,
e.g., $1, w_1,w_2$ form a minimal set of generators, and then
$\Ok_X$ is not a free $\Ok_Z$-module, since there is a non-trivial relation between
$w_1$ and $w_2$.

We claim that $\Ok_X$ has pure dimension. That is, we claim that there is no embedded
associated prime ideal at $(0,0)$; since $Z$ is irreducible, this is the same as saying that
$\J$ is primary with respect to $Z$.   To see the claim,
let $\phi$ and $\psi$ be functions such that
$\phi \psi$ is in $\J$ and $\psi$ is not in $\sqrt{\J}$. The latter assumption means,
in view of the Nullstellensatz,  that $\psi$ does not vanish identically on $Z$, i.e.,
$\psi=a(z) + \Ok(w)$, where $a$ does not vanish identically.  Since in particular
$\phi\psi$ must vanish on $Z$ it follows that $\phi=\Ok(w)$. It is now easy to see that
$\phi$ is in $\J$.  We conclude that $\J$ is primary.
\end{ex}

The pure-dimensionality of $\Ok_X$ can also be rephrased in the following way:
{\it If $\phi$ is holomorphic and is $0$ generically, then $\phi=0$.}
If we delete
the generator $w_1w_2$ from the definition of $\J$ in the example, then $\phi=w_1w_2$
is $0$ generically in $\Ok_\Omega/\J$ but is not identically zero. Thus
$\J$ then has an embedded primary ideal at $(0,0)$.

\begin{ex}\label{toto}
Let $\Omega=\C^2_{z,w}$ and $\J=(w^2)$ so that $Z=\{w=0\}$.
Then $1,w$ is a basis for
$\Ok_X=\Ok_{\C^2}/(w^2)$ so each function $\phi$ in $\Ok_X$ has a unique
representation $a_0(z)\otimes 1 +a_1(z)\otimes w$.  Let us consider the new
coordinates $\zeta=z-w, \eta=w$.  Then $\J=(\eta^2)$ and
since
$$
a_0(z) +a_1(z)w=a_0(\zeta+\eta)+a_1(\zeta+\eta)\eta=
a_0(\zeta)+(\partial a_0/\partial\zeta)(\zeta) \eta+a_1(\zeta)\eta+\J
$$
we have the representation
$a_0(\zeta)\otimes 1+(a_1(\zeta)+\partial a_0/\partial\zeta)(\zeta)\otimes\eta$
with respect to $(\zeta,\eta)$.
\end{ex}

More generally,   assume that, at a given point in $X_{reg}\subset\Omega$,
we have two different choices $(z,w)$ and
$(\zeta,\eta)$ of coordinates so that $Z=\{w=0\}=\{\eta=0\}$,
and bases $1, \ldots, w^{\alpha_{\nu-1}}$
and $1, \ldots, \eta^{\beta_{\nu-1}}$ for $\Ok_X$ as a free module over
$\Ok_Z$.  Then there is a
$\nu\times\nu$-matrix $L$ of holomorphic differential operators
so that if $(a_j)$ is any tuple in $(\Ok_Z)^\nu$ and $(b_j)=L(a_j)$, then
$
a_0\otimes 1+\cdots + a_{\nu-1}\otimes w^{\alpha_{\nu-1}}=
b_0\otimes 1+\cdots+ b_{\nu-1} \otimes \eta^{\beta_{\nu-1}}+\J.
$

\section{Smooth $(0,*)$-forms on a non-reduced space $X$}\label{boxer}

Let $i\colon X\to \Omega$ be a local embedding of $X$.  In order to define
the sheaf of smooth $(0,*)$-forms on $X$, in analogy with the reduced case,
we have to state which smooth $(0,*)$-forms $\Phi$ in $\Omega$ "vanish" on $X$, or more
formally, give a meaning to $i^*\Phi=0$.  We will see, cf.\ Lemma \ref{bura} below, that the suitable requirement is that locally on $X_{reg}$, $\Phi$ 
belongs to $\E_\Omega^{0,*} \J+\E_\Omega^{0,*}\bar \J_Z +\E_\Omega^{0,*} d\bar \J_Z$, where $\J_Z$ is the ideal sheaf defining $Z$.
However, it turns out to be more convenient to represent the sheaf
$\Kers i^*$ of such forms as the annihilator of certain residue currents, and this is the
path we will follow.  Moreover, these currents play a central role
themselves later on.


The following classical duality result is fundamental for this paper; see, e.g., \cite{Aext} for a discussion.
\begin{prop}\label{duality}
If $\J$ has pure dimension, then
\begin{equation}\label{kondor}
    \J= \ann_{\Ok_\Omega} \Homs(\Ok_\Omega/\J,\CH_\Omega^Z).
\end{equation}
\end{prop}
That is, $\phi$ is in $\J$ if and only if $\phi\mu=0$ for all $\mu$ in $\Homs(\Ok_\Omega/\J,\CH_\Omega^Z)$.
It is also well-known, see, e.g.,  \cite[Theorem~1.5]{Aext},  that
\begin{equation}\label{apanage}
\Homs(\Ok_\Omega/\J,\CH_\Omega^Z)\simeq  \Exts^p (\Ok_\Omega/\J,K_\Omega),
\end{equation}
so  $\Homs(\Ok_\Omega/\J,\CH_\Omega^Z)$ is a coherent analytic sheaf.
Locally we thus have a finite number of generators
$\mu^1,\ldots,\mu^m$.
In Example~\ref{exannJ}, we compute explicitly such generators for the ideal
$\J$ in Example~\ref{steira}.


\smallskip
Let $\xi$ be a smooth $(0,*)$-form in $\Omega$.
Without first giving meaning to $i^*$, we define the sheaf $\Kers i^*$ by
saying that $\xi$ is in $\Kers i^*$ if
$$
\xi\w\mu=0,  \quad \mu\in \Homs(\Ok_\Omega/\J,\CH^Z_\Omega).
$$
Notice that if $\xi$ is holomorphic, then, in view of the duality \eqref{kondor},
$\xi$ is in  $\Kers i^*$  if and only if $\xi$ is in $\J$.

\begin{df}
We define the sheaf of smooth $(0,*)$-forms on $X$ as
\begin{equation}\label{skrot1}
\E_X^{0,*}:=\E_\Omega^{0,*}/\Kers i^*.
\end{equation}
We will prove below that this sheaf is independent of the choice of embedding
and thus intrinsic on $X$.
\end{df}
Given $\phi$ in $\E_\Omega^{0,*}$, let $i^*\phi$
be its image in $\E_X^{0,*}$.
In particular, $i^* \xi = 0$ means that $\xi$ belongs to $\Kers i^*$,
which then motivates this notation.  
Notice that $\Kers i^*$ is a two-sided ideal in $\E_\Omega^{0,*}$,
i.e., if $\phi$ is in $\E_\Omega^{0,*}$ and $\xi$ is in $\Kers i^*$, then $\phi\w\xi$ and $\xi \w \phi$ are in $\Kers i^*$.
It follows that we have an induced wedge product on $\E_X^{0,*}$ such that
$$
i^*(\phi\w\xi)=i^*\phi\w i^*\xi.
$$

\begin{remark} It follows from Lemma~\ref{bura} below that in case $X=Z$ is reduced, then
$\xi$ is in $\Kers i^*$ if and only its pullback  to $X_{reg}$ vanishes.
Thus our definition of $\E_X^{0,*}$ is consistent with the usual one in that case.
\end{remark}


\begin{lma} \label{skot}
Using the notation of \eqref{embedding},
\begin{equation}\label{skata}
\iota_*\colon \Homs_{\Ok_{\widehat\Omega}}(\Ok_{\widehat\Omega}/\widehat\J,\W^Z_{\widehat \Omega})
\to
 \Homs_{\Ok_\Omega}(\Ok_\Omega/\J,\W_\Omega^Z)
\end{equation}
is an isomorphism.
\end{lma}

We can realize the mapping in \eqref{skata} as the tensor
product $\tau\mapsto \tau\w [w=0]$, where
$[w=0]$ is the Lelong current in $\Omega$  associated with the submanifold
$\{w=0\}$.

\begin{proof} To begin with,  $\iota_*$ maps
pseudomeromorphic $(\hat N,\hat p+\ell)$-currents with support on
$Z\subset\widehat\Omega$ to pseudomeromorphic
$(N,p+\ell)$-currents with support on
$Z\subset\Omega$. If, in addition,  $\tau$ has the SEP with respect
to $Z$, then $\iota_*\tau$ has, as well by \eqref{stare2prim}.
Moreover, if $\tau$ is annihilated by $\widehat\J$, then
$\iota_*\tau$ is annihilated by $\J=\widehat\J\otimes 1+(w)$.
Thus the mapping \eqref{skata} is well-defined,  and it is injective since
$\iota$ is injective.

Now assume that $\mu$ is in $\Homs(\Ok_\Omega/\J,\W^Z_\Omega)$.
Arguing as in the proof of Corollary~\ref{Wextrinsic}, we see that
$\mu=\iota_*\hat\mu$ for a current $\hat\mu$ in $\W^Z_{\widehat\Omega}$.
Since $\widehat\J=\iota^*\J$ and $\J\mu=0$, it follows that
 $\widehat\J\hat\mu=0$.
Thus \eqref{skata} is surjective.
\end{proof}

Since  $\iota_*$ is injective,
$\dbar\tau=0$ if and only if $\dbar\iota_*\tau=0$, and thus  we get

\begin{cor}\label{klas1}
Using the notation of \eqref{embedding},
\begin{equation}\label{skata2}
\iota_*\colon \Homs_{\Ok_{\widehat\Omega}}(\Ok_{\widehat\Omega}/\widehat\J,\CH^Z_{\widehat \Omega})
\to
 \Homs_{\Ok_\Omega}(\Ok_\Omega/\J,\CH_\Omega^Z)
\end{equation}
is an isomorphism.
\end{cor}

\begin{cor}\label{klas2}
Using the notation of \eqref{embedding}, 
\begin{equation}\label{skata4}
\iota^*\colon \E_{\Omega}^{0,*}/\Kers i^*\to\E_{\widehat\Omega}^{0,*}/\Kers j^*,
\end{equation}
is an isomorphism.
\end{cor}

\begin{proof} It follows immediately from \eqref{skata2} that the mapping
\eqref{skata4} is well-defined and injective. Given $\widehat{\xi}$ in $\E^{0,*}_{\widehat\Omega}$,
let $\xi=\widehat{\xi}\otimes 1$. Then $\iota^*\xi=\widehat{\xi}$ and so \eqref{skata4} is indeed
surjective as well.
\end{proof}

It follows from \eqref{skata4} and \eqref{skrot1}
that the sheaf  $\E_X^{0,*}$ is intrinsically defined on $X$.
Since $\dbar$ maps $\Kers i^*$ to $\Kers i^*$, we have a well-defined operator
$\dbar\colon\E_X^{0,*}\to\E_X^{0,*+1}$ such that $\dbar^2=0$.
Unfortunately the sheaf complex so obtained is not exact in general, see, e.g., \cite[Example~1.1]{AS}
for a counterexample already in the reduced case.

\subsection{Local representation on $X_{reg}$ of smooth forms}\label{gurka}
Recall that  $X_{reg}$ is the open subset of $X$, where the underlying reduced
space is smooth and $\Ok_X$ is Cohen-Macaulay.
Let us fix some point in $X_{reg}$, and assume that we have
local coordinates $(z,w)$ such that $Z = \{ w = 0 \}$.
We also choose generators $1,w^{\alpha_1},\dots,w^{\alpha_{\nu-1}}$ of $\Ok_X$
as a free $\Ok_Z$-module, which exist by Corollary~\ref{sorelle2},
and generators $\mu^1,\dots,\mu^m$ of $\Homs(\Ok_\Omega/\J,\CH^Z_\Omega)$.

Notice that for each smooth $(0,*)$-form $\Phi$ in $\Omega$,
$\Phi\mapsto \Phi\w \mu^\ell$  only depends on its class $\phi$ in
$\E^{0,*}_X$,  and  $\phi$ is in fact determined by these currents.
By  Proposition~\ref{parasol} each of these currents can (locally) be represented by
a tuple of currents in $\W_Z^{0,*}$. Putting all these tuples together, we get
a tuple in $(\W_Z^{0,*})^M$,
where $M = M_1 + \dots + M_m$ and $M_j$ is the number of indices in \eqref{paraply}
in the representation of $\mu^j$.

Recall from Corollary~\ref{sorelle2} that
$\phi$ in $\Ok_X$ has a unique representative
\begin{equation}\label{stor1}
\hat\phi=\hat\phi_0+\hat\phi_1\otimes w^{\alpha_1}+\cdots +\hat\phi_{\nu-1}\otimes w^{\alpha_{\nu-1}},
\end{equation}
where $\hat\phi_j$ are in $\Ok_Z$.
We thus have an $\Ok_Z$-linear morphism
\begin{equation} \label{eq:T}
T\colon (\Ok_Z)^\nu\to (\Ok_Z)^M.
\end{equation}
The morphism is injective by Proposition~\ref{duality},
and the  holomorphic matrix $T$ is therefore generically pointwise injective.

\begin{lma}\label{skottskada}
Each $\phi$ in $\E_X^{0,*}$ has a unique representation
\eqref{stor1}
 where $\hat\phi_j$ are in $\E_Z^{0,*}$.
\end{lma}

\begin{proof}
To begin with notice that a given smooth $\phi$ must have at least one such representation.
In fact, taking the finite Taylor expansion \eqref{paraply3}
we can forget about high order terms, since they must
annihilate all the $\mu^j$, and the terms $\bar{w}$ and $d\bar{w}$ annihilate all the $\mu^j$
as well since they are pseudomeromorphic with support on $\{ w = 0 \}$.
On the other hand, each $w^\alpha$ not in the set of
generators must be of the form
$$
w^\alpha=a_0+a_1\otimes w^{\alpha_1}+\cdots +a_{\nu-1}\otimes w^{\alpha_{\nu-1}}
+\J,
$$
and hence $\phi_\alpha\otimes w^\alpha$ is of the form \eqref{stor1}. Thus the representation exists.
To show uniqueness of the representation,
we assume that $\hat\phi$ is in $\Kers i^*$.  Then the tuple $(\hat\phi_j)$ is
mapped to $0$ by the matrix $T$, and since $T$ is generically pointwise injective we conclude
that each $\hat\phi_j$ vanishes.
 \end{proof}

By the above proof we get

\begin{lma}\label{bura}
A smooth $(0,*)$-form $\xi$ in $\Omega$ is in $\Kers i^*$ if and only if
$\xi$ is in $\E_\Omega^{0,*} \J+\E_\Omega^{0,*}\bar \J_Z
+\E_\Omega^{0,*} d\bar \J_Z$ on $X_{reg}$, where $\J_Z$ is the radical sheaf of $Z$.
\end{lma}

\begin{remark}
This is {\it not} the same as saying that $\xi$ is in
 $\E_\Omega^{0,*} \J+\E_\Omega^{0,*}\bar \J_Z
+\E_\Omega^{0,*} d\bar \J_Z$  at singular points. For a simple counterexample, consider
$\phi=x\bar y$ on the reduced space $Z=\{xy=0\}\subset\C^2$.

However, this can happen also when $Z$ is irreducible at a point. For example, the variety
$Z = \{ x^2 y - z^2 = 0  \} \subset \C^3$ is irreducible at $0$, but there exist points
arbitrarily close to $0$ such that $(Z,z)$ is not irreducible.
In this case, the ideal of smooth functions vanishing on $(Z,0)$ is strictly larger than
$\E_\Omega^{0,0} \J_{Z,0} + \E_\Omega^{0,0} \bar \J_{Z,0}$ see \cite[Proposition~9, Chapter~IV]{Nar}, and
\cite[Theorem~3.10, Chapter~VI]{MalBook}.
\end{remark}

\begin{remark}\label{getter}
It is easy to check that if we have the setting as in the discussion at the end of
Section \ref{get} but $(a_j)$ is instead a tuple in $\E_Z^{0,*}$, then
we can still define $(b_j)=L(a_j)$ if we consider the derivatives in
$L$ as Lie derivatives; in fact, since $a_j$ has no holomorphic differentials,
$L$ only acts on the smooth coefficients, and it is easy to check that
$a_0\otimes 1+\cdots + a_{\nu-1}\otimes w^{\alpha_{\nu-1}}$ and
$b_0\otimes 1+\cdots+ b_{\nu-1} \otimes \eta^{\beta_{\nu-1}}$
are equal modulo
$\E_\Omega^{0,*} \J+\E_\Omega^{0,*}\bar \J_Z
+\E_\Omega^{0,*} d\bar \J_Z$,
and thus define the same
element in $\E_X^{0,*}$.
\end{remark}

For future needs we prove in  Section~\ref{brakteat}:

 \begin{lma} \label{pointwise}
    The morphism $T$ is pointwise injective.
\end{lma}

We can thus choose a holomorphic matrix $A$ such that
\begin{equation}\label{parra}
0\to \Ok_Z^\nu\stackrel{T}{\to} \Ok_Z^M\stackrel{A}{\to} \Ok_Z^{M'}
\end{equation}
is pointwise exact, and we can also find holomorphic matrices $S$ and $B$
such that
\begin{equation} \label{Thomotopy}
    I = TS + BA.
\end{equation}

\section{Intrinsic $(n,*)$-currents  on $X$} \label{sect:currentsn}

In analogy with the reduced case we have the following definition when $X$ is possibly non-reduced.
\begin{df}
The sheaf $\Cu_X^{n,q}$ of $(n,q)$-currents on $X$ is the dual sheaf of $(0,n-q)$-test forms, i.e., forms in $\E_X^{0,n-q}$
with compact support.
\end{df}
Here, just as in the case of reduced spaces, cf.\ for example \cite[Section~4.2]{HeLi}, the space of smooth forms
$\E_X^{0,n-q}$ is equipped with the quotient topology induced by a local embedding.

More concretely, this means that given an embedding $i\colon X\to \Omega$, currents  $\psi$ in $\Cu_X^{n,q}$
precisely correspond to the $(N,N-n+q)$-currents
$\tau$ on $\Omega$  that vanish on $\Kers i^*$.
Since  $\Kers i^*$ is a two-sided ideal in  $\E_\Omega^{0,*}$ this holds
if and only if $\xi\w \tau=0$ for all
$\xi$ in $\Kers i^*$. It is natural to write $\tau=i_*\psi$ so that
$$
i_*\psi.\xi=\psi.i^*\xi.
$$
Clearly, we get
a mapping $\dbar\colon \Cu_X^{n,q}\to\Cu_X^{n,q+1}$ such that $\dbar^2=0$.

\smallskip

\begin{prop} \label{skatbo}
If $\tau$ is in $\W_\Omega^Z$ and $\J\tau=0$, then $\xi\w\tau=0$
for all smooth $\xi$ such that $i^*\xi=0$.
\end{prop}

\begin{proof} 
Because of the SEP it is enough
to prove that $\xi\w\tau=0$ on $X_{reg}$.
By assumption, $\J$ annihilates $\tau$, and by general properties of pseudomeromorphic
currents, since $\tau$ has support on $Z$, $\bar \J_Z$ and $d \bar \J_Z$ annihilate $\tau$.
Thus  the proposition  follows by Lemma~\ref{bura}.
\end{proof}

\begin{df}
An $(n,*)$-current $\psi$ on $X$ is in $\W_X^{n,*}$ if
$i_*\psi$ is in $\Homs(\Ok_\Omega/\J,\W_\Omega^Z)$.
 \end{df}

By definition we thus have the  isomorphism
\begin{equation}\label{hastverk}
i_*\colon \W_X^{n,*}\simeq \Homs(\Ok_\Omega/\J,\W_\Omega^Z).
\end{equation}
It follows from Lemma~\ref{skot} that $\W_X^{n,*}$ is intrinsically
defined.

\begin{remark} By Corollary~\ref{Wextrinsic}, this definition is consistent with the
previous definition of $\W_X^{n,*}$ when  $X$ is reduced.
We cannot define  $\PM_X^{n,*}$ in the analogous simple way, cf.\ Remark~\ref{soda}.
\end{remark}

\begin{df}\label{ainbusk}
    If $\psi$ is in $\W^{n,*}_X$ and $a$ is an almost semi-meromorphic $(0,*)$-current on $\Omega$ that is generically smooth on $Z$,
then the product $a \wedge \psi$ is
a current in $\W^{n,*}_X$ defined as follows:
By definition, $i_* \psi $ is in $\Homs(\Ok_\Omega/\J,\W^Z_\Omega)$
and by Proposition~\ref{asmW} and \eqref{korp},  one can define $a \wedge i_* \psi$ in $\Homs(\Ok_\Omega/\J,\W^Z_\Omega)$;
now  $a \wedge \psi$ is the unique current in $\W^{n,*}_X$ such that $i_*(a \wedge \psi) = a \wedge i_* \psi$.
\end{df}

By \eqref{proddef}, 
\begin{equation} \label{asmWn}
    a \wedge \psi = \lim_{\epsilon \to 0^+} \chi(|h|^2/\epsilon) a \wedge \psi
\end{equation}
  if $h$ cuts out the Zariski singular support of $a$.

\begin{df} We let $\Ba_X^n$ be the sheaf of $\dbar$-closed currents in $\W^{n,0}_X$.
\end{df}

This sheaf corresponds via $i_*$ to $\dbar$-closed currents in $\Homs(\Ok_\Omega/\J,\W_\Omega^Z)$
so we have the isomorphism
\begin{equation}\label{hastverk2}
    i_*\colon \Ba_X^n \simeq \Homs(\Ok_\Omega/\J,\CH_\Omega^Z).
\end{equation}
When $X$ is reduced $\Ba_X^n$ is the sheaf of
$(n,0)$-forms that are $\dbar$-closed in the
Barlet-Henkin-Passare sense.
Let $\mu^1,\ldots,\mu^m$ be a set of generators for the $\Ok_\Omega$-module
$\Homs(\Ok_\Omega/\J,\CH_\Omega^Z)$.  They correspond via \eqref{hastverk2}
to a set of generators
$h^1,\ldots, h^m$ for the $\Ok_X$-module $\Ba_X^n$.

\smallskip

We will also need a definition of $\PM_X^{n,*}$.
Let $\F_X$  be the subsheaf of $\Cu_X^{n,*}$ of $\tau$ such that $i_*\tau$ is in $\PM_\Omega^Z$.  If $\tau$ is a section of $\F_X$
and $W$ is a subvariety of some open subset of $Z$, then $\1_Wi_*\tau$ is in $\PM_\Omega^Z$, and by \eqref{stare2}, $\1_W i_* \tau$ is annihilated by $\Kers i^*$.
Hence  we can define $\1_W\tau$ as the unique current in $\F_X$ such that $i_*\1_W\tau=\1_Wi_*\tau$.
Clearly, $\1_W\tau$ has support on $W$  and it is easily checked that the  computational rule \eqref{stare2} holds also in $\F_X$.
Moreover, $\F_X$ is closed under
$\dbar$ since  $\PM_\Omega^Z$ is.

\begin{df} \label{polestar}
    The sheaf $\PM_X^{n,*}$ is the smallest subsheaf of $\F_X$ that contains $\W_X^{n,*}$ and is closed
under $\dbar$ and multiplication by $\1_W$ for all germs $W$ of subvarieties of $Z$.
\end{df}

In view of Proposition~\ref{trilsk} this definition coincides with the usual definition in case $X$ is reduced.
It is readily checked that the dimension principle holds for $\F_X$, and hence it also holds
for the (possibly smaller) sheaf $\PM_X^{n,*}$, and in addition, \eqref{stare2} holds
for forms $\xi$ in $\E_X^{0,*}$ and $\tau$ in $\PM_X^{n,*}$.

\section{Structure form on $X$}\label{prut}

Let $i\colon X\to \Omega \subset \C^N$ be a local embedding as before, let $p = N-n$ be the codimension of $X$, and let
$\J$ be the associated ideal sheaf on $\Omega$.
In a slightly smaller set, still denoted $\Omega$, there is  a free
resolution
\begin{equation}\label{karvupplosn}
0 \to \hol(E_{N_0}) \stackrel{f_{N_0}}{\longrightarrow} \cdots \stackrel{f_3}{\longrightarrow} \hol(E_2)
\stackrel{f_2}{\longrightarrow} \hol(E_1) \stackrel{f_1}{\longrightarrow} \hol(E_0)
\end{equation}
of $\hol_{\Omega}/\J$;
here  $E_k$ are trivial vector bundles over $\Omega$ and $E_0$ is the
trivial line bundle. This resolution induces a complex of  vector bundles
\begin{equation}\label{VBkomplex}
0 \to E_{N_0} \stackrel{f_{N_0}}{\longrightarrow} \cdots \stackrel{f_3}{\longrightarrow} E_2
\stackrel{f_2}{\longrightarrow} E_1 \stackrel{f_1}{\longrightarrow} E_0
\end{equation}
that is pointwise exact outside $Z$. Let $X_k$ be the set where $f_k$ does not have optimal rank. Then
\begin{equation*}
\cdots \subset X_{k+1} \subset X_k \subset \cdots \subset X_{p+1}\subset X_p =\cdots =X_1= Z;
\end{equation*}
these sets are independent of the choice of resolution and  thus invariants of $\hol_{\Omega}/\J$.
Since $\Ok_\Omega/\J$ has {\it pure} codimension $p$,
\begin{equation}\label{BEsats}
\codim X_k \geq k+1, \quad \mbox{for} \quad k\geq p+1,
\end{equation}
see  \cite[Corollary 20.14]{Eis}. Thus there is a free resolution \eqref{karvupplosn}
if and only if $X_k=\emptyset$ for $k>N_0$.
Unless $n = 0$ (which is not interesting in relation to the $\dbar$-equation),
we can thus choose the resolution so that $N_0\le N-1$.
The variety $X$ is Cohen-Macaulay at a point $x$, i.e., the sheaf $\Ok_\Omega/\J$ is Cohen-Macaulay at $x$,
if and only if $x\notin X_{p+1}$.
Notice that $Z\setminus (X_{reg})_{red}=Z_{sing}\cup X_{p+1}$.
The sets $X_k$ are  independent
of the choice of embedding, see \cite[Lemma~4.2]{AWsemester},
and  are thus intrinsic subvarieties of  $Z=X_{red}$,
and they reflect the complexity of the singularities of $X$.

\smallskip

Let us now choose Hermitian metrics on the bundles $E_k$. We then refer to
\eqref{karvupplosn} as a {\it Hermitian resolution} of $\Ok_\Omega/\J$ in $\Omega$.
In $\Omega\setminus X_k$ we have a well-defined vector bundle
morphism $\sigma_{k+1}\colon E_k\to E_{k+1}$, if we require that  $\sigma_{k+1}$ vanishes on
$(\Image f_{k+1})^\perp$, takes values in $(\Kers f_{k+1})^\perp$, and that $f_{k+1}\sigma_{k+1}$ is
the identity on $\Image f_{k+1}$.   Following  \cite[Section 2]{AW1} we  define  smooth
$E_k$-valued forms
\begin{equation}\label{ukdef}
u_k=(\dbar\sigma_k)\cdots(\dbar\sigma_2)\sigma_1=\sigma_k(\dbar\sigma_{k-1})\cdots(\dbar\sigma_1)
\end{equation}
in $\Omega\setminus  X$; for the second equality, see \cite[(2.3)]{AW1}. We have that
$$
f_1u_1=1, \quad f_{k+1}u_{k+1}-\dbar u_k=0,\quad  k\ge 1,
$$
in $\Omega\setminus  X$.
If $f:=\oplus f_k$ and $u:=\sum u_k$,  then these relations can be written
economically as $\nabla_f u=1$, where
$\nabla_f:= f-\dbar$.
To make the algebraic machinery  work properly one has to introduce a superstructure on the
bundle $E=:\oplus E_k$ so that vectors in $E_{2k}$  are even and vectors in  $E_{2k+1}$ are odd;
hence $f$, $\sigma:=\oplus\sigma_k$,  and  $u:=\sum u_k$  are odd.
For details, see \cite{AW1}.
It turns out that $u$ has a (necessarily unique)
almost semi-meromorphic extension $U$ to $\Omega$. The residue current
$R$ is  defined by the relation
\begin{equation}\label{potatis}
\nabla_f U=1-R.
\end{equation}
It follows directly that $R$ is $\nabla_f$-closed.
In addition, $R$ has support on $Z$ and is a sum $\sum R_k$, where $R_k$ is
a \pmm $E_k$-valued current of bidegree $(0,k)$. It follows from  the dimension
principle that $R=R_p+R_{p+1}+\cdots +R_N$. If we choose a free  resolution that
ends at level $N-1$, then $R_N=0$.
If $X$ is Cohen-Macaulay and $N_0=p$ in \eqref{karvupplosn}, then $R=R_p$,
and the $\nabla_f$-closedness implies that $R$ is $\debar$-closed.

If $\phi$ is in $\J$ then $\phi R=0$  and in fact, $\J=\ann R$,
see \cite[Theorem~1.1]{AW1}.


\begin{remark}\label{krokus}
In case $\J$ is generated by the single non-trivial function $f$, then we  have the free  resolution
$0\to \Ok_\Omega \stackrel{f}{\to} \Ok_\Omega\to\Ok_\Omega/(f)\to 0$;
thus $U$ is just the
principal value current $1/f$ and $R=\dbar(1/f)$.
More generally, if $f = (f_1,\dots,f_p)$ is a complete intersection, then
\begin{equation*}
    R = \dbar \frac{1}{f_p} \wedge \dots \wedge \dbar\frac{1}{f_1},
\end{equation*}
where the right hand side is the so-called Coleff-Herrera product of $f$,
see for example \cite[Corollary~3.5]{ACH}.
\end{remark}

There are almost semi-meromorphic $\alpha_k$ in $\Omega$, cf.\ \cite[Section~2]{AW1} and the proof of \cite[Proposition~3.3]{AS}, that are
smooth outside  $X_k$,  such that
\begin{equation}\label{plast}
R_{k+1}=\alpha_{k+1} R_k
\end{equation}
outside $X_{k+1}$ for $k\ge p$.  In view of \eqref{BEsats} and the dimension principle,
$\1_{X_{k+1}}R_{k+1}=0$ and hence \eqref{plast} holds across $X_{k+1}$, i.e., $R_{k+1}$ is
indeed equal to the product $\alpha_{k+1} R_k$ in the sense of Proposition~\ref{ball}.
In particular, it follows that $R_k$ has the SEP with respect to $Z$.

In this section, we let $(z_1,\dots,z_N)$ denote coordinates on $\C^N$,
and let $dz := dz_1 \w \cdots \w dz_N$.

\begin{lma}\label{bmy}
There is a matrix of almost semi-meromorphic currents $b$ such that
\begin{equation}\label{plast2}
R\w dz=b\mu,
\end{equation}
where $\mu$ is a tuple of currents in $\Homs(\Ok_\Omega/\J,\CH_\Omega^Z)$.
\end{lma}

\begin{proof}
As in \cite[Section~3]{AS}, see also \cite[Proposition~3.2]{Sz},
one can prove that $R_p=\sigma_F\mu$, where $\mu$ is a tuple of
currents in $\Homs(\Ok_\Omega/\J,\CH_\Omega^Z)$ and $\sigma_F$ is an almost
semi-meromorphic current that is smooth outside $X_{p+1}$.

Let $b_p=\sigma_F$ and $b_k=\alpha_k\cdots\alpha_{p+1}\sigma_F$ for
$k\ge p+1$.  Then each $b_k$ is almost semi-meromorphic,
cf.\ \cite[Section 4.1]{AW3}.  In view of \eqref{plast} we have that
$R_k=b_k\mu$ outside $X_{p+1}$ since $b_k$ is smooth there. It follows by the
SEP that it holds across $X_{p+1}$ as well since $R_k$ has the SEP with respect to $Z$.
We then take $b=b_p+b_{p+1}+\cdots$.
\end{proof}

By Proposition~\ref{asmW} we get

\begin{cor}\label{rattmuff}
The current $R\wedge dz$ is in $\Homs(\Ok_\Omega/\J,\W_\Omega^Z)$.
\end{cor}

From Lemma \ref{bmy},  Corollary~\ref{rattmuff}, \eqref{hastverk},
and \eqref{hastverk2} we get
the following analogue to \cite[Proposition~3.3]{AS}:

\begin{prop}
    Let \eqref{karvupplosn} be a Hermitian resolution of $\Ok_\Omega/\J$ in $\Omega$,
    and let $R$ be the associated residue current.
    Then there exists a (unique) current
    $\omega$ in $\W_X^{n,*}$ such that
\begin{equation}\label{struktur}
i_*\omega=R\w dz.
\end{equation}
There is a matrix $b$ of almost semi-meromorphic $(0,*)$-currents in $\Omega$,
smooth outside of $X_{p+1}$,
and a tuple $\vartheta$ of currents in $\Ba^n_X$ such that
\begin{equation} \label{omegatheta}
\omega = b \vartheta.
\end{equation}
More precisely, $\omega=\omega_0+\omega_1+\cdots+\omega_n$\footnote{In \cite[Proposition~3.3]{AS}, the sum ends with $\omega_{n-1}$ instead of $\omega_n$,
which, as remarked above, one can indeed assume when $n \geq 1$ and the resolution is chosen to be of length $\leq N-1$.},
where $\omega_k \in \W^{n,k}(X, E_{p+k})$, and if $f^j:=f_{p+j}$, then
\begin{equation}\label{muff2}
    f^0\omega_0=0, \quad f^{j+1}\omega_{j+1}-\dbar\omega_j=0 \text{, for } j\geq 0.
\end{equation}
\end{prop}

We will also use the short-hand notation $\nabla_f \omega=0$.
As in the reduced case, following \cite{AS},  we say that $\omega$ is a {\it structure form}
for $X$.
The products in \eqref{omegatheta} are defined according to Definition \ref{ainbusk}.

\begin{remark}\label{mars}
Recall that $X_{p+1}=\emptyset$ if $X$ is Cohen-Macaulay, so in that case
$\omega=b\vartheta$,  where $b$ is smooth. If we take a free resolution
of length $p$, then $\omega = \omega_0$, and $\dbar \omega_0 = f^1 \omega_1 = 0$,
so $\omega$ is in $\Ba^n_X$.
\end{remark}

\begin{remark}
    If $X=\{f=0\}$ is a reduced hypersurface in  $\Omega$, then
    $R=\dbar(1/f)$  and $\omega$ is the classical Poincar\'e residue form on $X$ associated with $f$,
    which is a meromorphic form on $X$. More generally, if $X$ is reduced, since forms in $\Ba^n_X$
    are then meromorphic, by \eqref{omegatheta}, $\omega$ can be represented by almost semi-meromorphic forms on $X$.

    We now consider the case when $X$ is non-reduced. We recall that a differential operator is a Noetherian operator for an ideal $\J$
    if $\mathcal{L}\varphi \in \sqrt{\J}$ for all $\varphi \in \J$.
    It is proved by Bj\"ork, \cite{BjAbel}, see also \cite[Theorem~2.2]{Sz}, that if $\mu \in \Homs(\Ok_\Omega/\J,\CH^Z_\Omega)$,
    then there exists a Noetherian operator $\mathcal{L}$ for $\J$ with meromorphic coefficients
    such that the action of $\mu$ on $\xi$ equals the integral of $\mathcal{L} \xi$ over $Z$.
    By \eqref{hastverk2}, the action of $h$ in $\Ba^n_X$ on $\xi$ in $\E^{0,*}_X$ can then be expressed as
        $$h . \xi = \int_Z \mathcal{L} \xi.$$
    One can then verify using this formula and \eqref{omegatheta} that the action of the structure form $\omega$
    on a test form $\xi$ in $\E^{0,*}_X$ equals
    \begin{equation*}
        \omega . \xi = \int_Z \tilde{\mathcal{L}} \xi,
    \end{equation*}
    where $\tilde{\mathcal{L}}$ is now a tuple of Noetherian operators for $\J$ with almost semi-meromorphic coefficients, cf.\ \cite[Section~4]{Sz}.
\end{remark}

Notice that \eqref{karvupplosn}  gives rise to the dual Hermitian complex
\begin{equation}\label{dualkomplex}
0\to \Ok(E_0^*) \stackrel{f^*_1}{\to}\cdots \to \Ok(E^*_{p-1})\stackrel{f^*_p}{\to}
\Ok(E_p^*)\stackrel{f^*_{p+1}}\longrightarrow\cdots .
\end{equation}
Let $\xi = \xi_0 \w dz$ be a holomorphic section of the sheaf
$$
\Homs(E_p,K_\Omega)\simeq \Ok(E^*_p)\otimes \Ok(K_\Omega)
$$
such that $f^*_{p+1}\xi_0=0$. Then
$\dbar(\xi_0 \omega_0)=\pm\xi_0\dbar\omega_0=\pm\xi_0 f_{p+1}\omega_1=
\pm (f^*_{p+1}\xi_0)\omega_1=0$, so that $\xi_0\omega_0$ is in $\Ba_X^n$.
Moreover, if $\xi_0=f^*_p\eta$ for $\eta$ in $\Ok(E^*_{p-1})$, then $\xi_0\omega_0=
f^*_p\eta\omega_0=\pm\eta f_p\omega_0=0$.
We thus have a sheaf mapping
\begin{equation}\label{greta}
\Ho^p(\Homs(E_\bullet,K_\Omega))\to\Ba_X^n, \quad \xi_0 \w dz \mapsto \xi_0\omega_0.
\end{equation}

\begin{prop}\label{takyta}
The mapping \eqref{greta} is an isomorphism, which establishes an intrinsic isomorphism
\begin{equation}\label{pucko}
\Exts^p(\Ok_\Omega/\J,K_\Omega)\simeq \Ba_X^n.
\end{equation}
\end{prop}

\begin{proof}
If $h$ is in $\Ba_X^n$, then $i_*h$ is in $\Homs(\Ok_\Omega/\J,\CH_\Omega^Z)$.
We have mappings
\begin{equation}\label{utter}
\Ho^p(\Homs(E_\bullet,K_\Omega))
\to\Ba_X^n \stackrel{\simeq}{\to} \Homs(\Ok_\Omega/\J,\CH^Z_\Omega),
\end{equation}
where the first mapping is \eqref{greta}, and the second is $h \mapsto i_*h$.
In view of \eqref{struktur}, the composed mapping is $\xi = \xi_0 \w dz \mapsto \xi R_p = \xi_0 R_p \wedge dz$
\footnote{There is a superstructure involved, with respect to which $R_p$ has even degree, and therefore $dz \w R_p = R_p \w dz$,
explaining the lack of a sign in the last equality, see \cite{AW1} or \cite{AS}.}.
This mapping is an intrinsic isomorphism
$$\Exts^p(\Ok_\Omega/\J,K_\Omega)\simeq \Homs(\Ok_\Omega/\J,\CH^Z_\Omega)$$
according to \cite[Theorem~1.5]{Aext}.
It follows that \eqref{greta} also establishes an intrinsic isomorphism.
\end{proof}

In particular it follows that  $\Ba_X^n$ is coherent, and we have:

\smallskip
\noindent
{\it If $\xi^1, \ldots,\xi^m$ are  generators of
$\Ho^p(\Homs(E_\bullet^*,K_\Omega)))$, where $\xi^\ell = \xi^\ell_0 \w dz$, then
$h^\ell:=\xi^\ell_0 \omega_0, \ \ell=1,\ldots,m$,
generate the $\Ok_X$-module $\Ba_X^n$,
and $\mu^\ell= i_* h^\ell = \xi^\ell R_p$  generate
the $\Ok_\Omega$-module $\Homs(\Ok_\Omega/\J,\CH_\Omega^Z)$.}

\begin{remark} \label{remext} The isomorphism
\begin{equation}\label{rubel}
\Ho^p(\Homs(E_\bullet,K_\Omega)) \stackrel{\simeq}{\to}\Homs(\Ok_\Omega/\J,\CH^Z_\Omega)
\end{equation}
was well-known since long ago, the contribution in \cite{Aext} was the realization $\xi \mapsto \xi R_p$.
\end{remark}


We give here an example where we can explicitly compute generators of $\Homs(\Ok_\Omega/\J,\CH^Z_\Omega)$.

\begin{ex} \label{exannJ}
    Let $\J$ be as in Example~\ref{steira}. We claim that
    $\Homs(\Ok_\Omega/\J,\CH^Z_\Omega)$ is generated by
    \begin{equation*}
        \mu_1 := \dbar\frac{1}{w_1} \wedge \dbar \frac{1}{w_2} \w dz \w dw \text{ and } \mu_2 := \left( z_1 \dbar \frac{1}{w_1^2} \wedge \dbar \frac{1}{w_2} + z_2 \dbar \frac{1}{w_1} \wedge \dbar \frac{1}{w_2^2}\right) \w dz \w dw.
    \end{equation*}
    In order to prove this claim, we use the comparison formula for residue currents from \cite{LarComp},
    which states that if $\Ok(F_\bullet)$ and $\Ok(E_\bullet)$
    are free resolutions of $\Ok_\Omega/\mathcal{I}$ and $\Ok_\Omega/\J$, respectively, where $\mathcal{I}$ and $\J$ have codimension
    $\geq p$, and $a : F_\bullet \to E_\bullet$ is a morphism of complexes, then there exists a $\Homs(F_0,E_{p+1})$-valued current
    $M_{p+1}$ such that $R^E_p a_0 = a_p R^F_p + f_{p+1} M_{p+1}$. If $\xi$ is in $\Kers f_{p+1}^*$, we thus get that
    \begin{equation} \label{jamfor}
        \xi R^E_p a_0 = \xi a_p R^F_p.
    \end{equation}
    We will apply this with $\Ok_\Omega(E_\bullet)$ as the free resolution
    \begin{equation*}
        0 \to \Ok_\Omega \stackrel{f_3}{\longrightarrow} \Ok_\Omega^4 \stackrel{f_2}{\longrightarrow} \Ok_\Omega^4 \stackrel{f_1}{\longrightarrow}
        \Ok_\Omega \to \Ok_\Omega/\J \to 0,
    \end{equation*}
    where
    \begin{eqnarray*}
        f_3 = \left[ \begin{array}{c} w_2 \\ -w_1 \\ z_2 \\ -z_1 \end{array} \right] \text{, }
        f_2 = \left[ \begin{array}{cccc} z_2 & 0 & -w_2 & 0 \\ -z_1 & z_2 & w_1 & -w_2 \\ 0 & -z_1 & 0 & w_1 \\ -w_1 & -w_2 & 0 & 0 \end{array} \right] \\
            \text{ and }
            f_1 = \left[ \begin{array}{cccc} w_1^2 & w_1 w_2 & w_2^2 & z_2 w_1 - z_1 w_2 \end{array} \right],
    \end{eqnarray*}
    and the Koszul complex  $(F,\delta_{\mathbf{w}^2})$ generated by
 $\mathbf{w}^2 := (w_1^2,w_2^2)$, which is a free resolution of $\Ok/(w_1^2,w_2^2)$.
    We then take the morphism of complexes $a : F_\bullet \to E_\bullet$ given by
    \begin{equation*}
        a_2 = \left[ \begin{array}{c} 0 \\ 0 \\ w_2 \\ w_1 \end{array} \right] \text{, }
        a_1 = \left[ \begin{array}{cc} 1 & 0 \\ 0 & 0 \\ 0 & 1 \\ 0 & 0 \end{array} \right]
            \text{ and }
            a_0 = \left[ \begin{array}{c} 1 \end{array} \right].
    \end{equation*}
    Since the current $R^F_2$ is equal to the Coleff-Herrera product $\dbar(1/w_1^2) \wedge \dbar(1/w_2^2)$, cf.\ Remark~\ref{krokus}, we thus get
    by \eqref{jamfor} and Remark~\ref{remext} that $\Homs(\Ok_\Omega/\J,\CH^Z_\Omega)$ is generated by
    \begin{equation*}
        (\Kers f_3^*) a_2 \dbar \frac{1}{w_1^2} \wedge \dbar \frac{1}{w_2^2}.
    \end{equation*}
    A straightforward calculation gives the generators $\mu_1$ and $\mu_2$ above.
\end{ex}

\subsection{Proof of Lemma~\ref{pointwise}}\label{brakteat}
Since $T$ is generically injective, it is clearly injective if  $n=0$.  We are going to
reduce to this case. Fix the point $0\in Z$ and let $\mathcal{I}$ be the ideal generated by
$z=(z_1,\ldots,z_n)$.

Let $\Ok(E_\bullet)$ be a free Hermitian resolution of $\Ok_\Omega/\J$ of minimal
length $p=N-n$ at $0$ and let $R^E$ be the associated residue current.  Recall that
the  canonical isomorphism \eqref{rubel} is
realized by $\xi\mapsto \xi R_p^E$.  Let $F_\bullet$ be the Koszul complex generated by $z$;
then $\Ok(F_\bullet)$ is a free resolution of $\Ok_\Omega/\mathcal{I}$.
Since $\J$ and $\mathcal{I}$ are Cohen-Macaulay and intersect properly in $\Omega$, the
complex   $\Ok_\Omega( (E\otimes F)_\bullet)$ is a
free resolution of $\Ok_\Omega/(\J+\mathcal{I})$, and the corresponding
residue current is
$$
R^{E\otimes F}_N = R^E_p \wedge R^F_n
$$
according to \cite[Theorem~4.2]{Astrong}.  From \cite[Theorem~1.5]{Aext} again it follows that
the canonical isomorphism
$$
    \Ho^N(\Homs((E\otimes F)_\bullet,K_\Omega))\to \Homs(\Ok_\Omega/(\J+\mathcal{I}),\CH^{\{0\}}_\Omega)
$$
is given by $\eta\mapsto \eta R^{E\otimes F}_N$.

Let  $\mu^1,\ldots,\mu^m$ be a minimal set of generators for the $\Ok_\Omega$-module $ \Homs(\Ok_\Omega/\J,\CH^Z_\Omega)$
at $0$. Then $\mu^j=\xi^jR^E_p$, where $\xi^j$ is a minimal set of generators for \newline
$\Ho^p(\Homs(E_\bullet,K_\Omega))$.
Notice that
$$
\Ho^N(\Homs((E\otimes F)_\bullet,K_\Omega))=\Ho^p(\Homs(E_\bullet,K_\Omega))
\otimes_{\Ok} \Ho^n(\Homs(F_\bullet,\Ok_\Omega)).
$$
Since $\Ho^n(\Homs(F_\bullet,\Ok_\Omega))$ is generated by $1$, it follows that
$\Ho^N(\Homs( (E\otimes F)_\bullet,K_\Omega))$ is generated by $\xi^j\otimes 1$. We conclude that
$\Homs(\Ok_\Omega/(\J+\mathcal{I}),\CH^{\{0\}}_\Omega)$ is generated by
$\xi^j\otimes 1\cdot R^E_p\w R^F_n=\mu^j\w \mu^z$, $j=1,\ldots,m$,
where $R^F_n = \mu^z = \dbar(1/z^{\1})$.


If $1, \ldots, w^{\alpha_{\nu-1}}$ is a basis for $\Ok_\Omega/\J$ as an $\Ok_Z$-module, then it is also
a  basis for $\Ok_{X_0} := \Ok_\Omega/(\J+\mathcal{I})$ as a module over  $\Ok_{\{0\}}\simeq\C$.  Since $\phi \dbar(1/z^{\1}) = \phi(0,\cdot)\dbar(1/z^{\1})$
we have that
$$
\phi(z,w)\mu^j \wedge \mu^z=\phi(z,w)\sum a^j_\ell(z)\dbar\frac{1}{w^{\ell+\1}}\w \dbar \frac{1}{z^{\1}}=
\phi(0,w)\sum a^j_\ell(0)\dbar\frac{1}{w^{\ell+\1}}\w \dbar \frac{1}{z^{\1}}.
$$
The morphism constructed in \eqref{eq:T} for $X_0$ instead of $X$ is then
$T_0 = T(0)$, where $T$ is the morphism \eqref{eq:T} for $X$.
Thus $T(0)$ is injective.

\section{The intrinsic sheaf $\W_X^{0,*}$ on $X$}\label{oxet}

Our aim is to find a fine resolution of $\Ok_X$ and
since the complex \eqref{dolb} is not exact in general when $X$ is singular
we have to consider larger fine sheaves; we first define sheaves $\W_X^{0,*}\supset\E_X^{0,*}$ of $(0,*)$-currents.
Given a local embedding $i\colon X\to \Omega$ at a point on $X_{reg}$
and local coordinates $(z,w)$ as before, it is natural, in view of Lemma
\ref{skottskada},
to require that an element in $\W_X^{0,*}$ shall have a unique representation
 \begin{equation} \label{eq:W0Xreg}
 \phi = \widehat{\phi}_0 \otimes 1 + \widehat{\phi}_1 \otimes w^{\alpha_1} + \dots + \widehat{\phi}_{\nu-1} \otimes w^{\alpha_{\nu-1}},
\end{equation}
where $\widehat{\phi}_j$ are in $\W^{0,*}_Z$.
In view of Remark \ref{getter} we should expect that the same transformation
rules hold as for smooth $(0,*)$-forms.  In particular it is then necessary that
$\W_Z^{0.*}$ is closed under the action of holomorphic differential operators,
which in fact is true, see Proposition~\ref{stek} below.
We must also define a reasonable extension of these sheaves
across $X_{sing}$.   Before we present our formal definition we make
a preliminary observation.


\begin{lma}\label{motor}
If $\phi$ has the form \eqref{eq:W0Xreg} and  $\tau$ is in
$\Homs(\Ok_\Omega/\J,\CH^Z_\Omega)$,  expressed in the form
\eqref{paraply}, then
\begin{equation}\label{ansgar}
\phi \wedge \tau:= \sum_i \sum_{\gamma \geq \alpha_i} \widehat{\phi}_i \wedge \tau_\gamma \wedge dz \otimes\dbar\frac{dw}{w^{\gamma-\alpha_i+\1}}
\end{equation}
is in  $\Homs(\Ok_\Omega/\J,\W^Z_\Omega)$.
\end{lma}

\begin{proof}
The right hand side defines a current in
$\W_\Omega^Z$ since $\widehat{\phi}_i$ are in $\W^{0,*}_Z$ and
$\tau_{\gamma}$ are in $\Ok_Z$.  We have to prove that it is annihilated by $\J$.
Take $\xi$ in $\J$.
On the subset of $Z$ where $\widehat{\phi}_0,\dots,\widehat{\phi}_{\nu-1}$
are all smooth, $\phi \wedge \tau$, as defined above,
is just multiplication of the smooth form $\phi$ by $\tau$, and thus
$\xi \phi\w\tau = 0$ there.   We have  a unique representation
$$
\xi \phi\w\tau=\sum_{\ell\ge 0} a_\ell(z)\w dz\otimes\dbar\frac{dw}{w^{\ell+\1}},
$$
with  $a_{\ell}$  in $\W_Z^{0,*}$.  Since  $a_\ell$ vanish on the
set where all $\widehat{\phi}_j$ are smooth, we conclude from Proposition~\ref{Wsmooth} that $a_\ell$ vanish identically.
It follows that $\xi \phi\w\tau = 0$.
\end{proof}

If $\phi$ has the form \eqref{eq:W0Xreg} in a \nbh of some point $x\in X_{reg}$
and $h$ is in $\Ba^n_X$,
then we get an element $\phi\w h$ in $\W_X^{n,*}$
defined by $i_*(\phi\w h)=\phi\w i_*h$.
It follows that $\phi$ in this way defines an element in
$\Homs_{\Ok_X}(\Ba^n_X, \W_X^{n,*})$.  This sheaf
is global and invariantly defined and so we can make the following
global definition.

\begin{df} \label{motor1}
$\W_X^{0,*}=\Homs_{\Ok_X}(\Ba^n_X, \W_X^{n,*})$.
\end{df}

If $\phi$ is in $\W^{0,*}_X$ and $h$ is in $\Ba^n_X$, we consider $\phi(h)$ as the product
of $\phi$ and $h$, and sometimes write it as $\phi \wedge h$.

Since  $\W_X^{n,*}$ are $\E^{0,*}_X$-modules,
$\W_X^{0,*}$ are as well.  Before we investigate these sheaves further, we
give some motivation for the definition.
First notice that we have a natural injection, cf.\ Proposition \ref{duality},
\begin{equation} \label{roos}
        \Ok_X \to \Homs(\Ba^n_X,\Ba^n_X), \quad \phi\mapsto (h\mapsto \phi h).
\end{equation}

\begin{thm}\label{roossats}
The mapping \eqref{roos} is an isomorphism in the Zariski-open
subset of $X$ where it is $S_2$.
\end{thm}

This is the subset of $X$ where $\codim X_k\ge k+2$, $k\ge p+1$,
cf.\ Section \ref{prut}.
Thus it contains all points $x$ such that $\Ok_{X,x}$ is  Cohen-Macaulay.  In particular, \eqref{roos} is an isomorphism in $X_{reg}$.

Theorem~\ref{roossats} is a consequence of the results in \cite{LExpl}.
If $X$ has pure dimension $p$, there is an injective mapping
\begin{equation} \label{Lisom1}
        \Ok_X \to \Homs(\Exts^p(\Ok_X,K_\Omega),\CH^Z_\Omega),
\end{equation}
which by \cite[Theorem~1.2 and Remark~6.11]{LExpl} is an isomorphism if and only if $\Ok_X$ is $S_2$.
Since the image of such a morphism must be annihilated
by $\J$ by linearity, it is indeed a morphism
\begin{equation} \label{Lisom2}
        \Ok_X \to \Homs(\Exts^p(\Ok_X,K_\Omega),\Homs(\Ok_\Omega/\J,\CH^Z_\Omega)).
\end{equation}
In view of \eqref{apanage}  and  \eqref{hastverk2}, \eqref{Lisom2} corresponds
to a morphism $\Ok_X \to \Homs(\Ba^n_X,\Ba^n_X)$, and the fact that it is the morphism
\eqref{roos} is a rather simple consequence of the definition of the morphism \eqref{Lisom1} in \cite[(6.9)]{LExpl}.

\medskip

As mentioned in the introduction,
Theorem \ref{roossats} can be seen as a reformulation
of a  classical  result of Roos, \cite{Roos},
which is the same statement about the injection
\begin{equation} \label{Risom}
    \Ok_\Omega/\J \to \Exts^p(\Exts^p(\Ok_\Omega/\J,K_\Omega),K_\Omega);
\end{equation}
here we assume that the ideal has pure dimension. The equivalence of the morphisms \eqref{Lisom1} and \eqref{Risom}
is discussed in \cite[Corollary~1.4]{LExpl}.

Let us now consider the case when $X$ is reduced. Since sections of $\Ba^n_X$ are meromorphic,
see \cite[Example~2.8]{AS}, and thus almost semi-meromorphic and generically smooth,
by Proposition~\ref{asmW} (with $Z = X = \Omega)$ we can extend \eqref{roos}
to a morphism
\begin{equation} \label{roosW}
    \W^{0,*}_X \to \Homs(\Ba^n_X,\W^{n,*}_X).
\end{equation}

\begin{lma}
    When $X$ is reduced \eqref{roosW} is an isomorphism.
\end{lma}

Thus Definition \ref{motor1} is consistent with the previous definition
of $\W^{0,*}_X$ when $X$ is reduced.

\begin{proof}
Clearly each $\phi$ in $\W^{0,*}_X$ defines an element $\alpha$ in
$\Homs(\Ba^n_X,\W^{n,*}_X)$ by $h\mapsto \phi\w h$. If we apply
this to a generically nonvanishing $h$ we see by the SEP that \eqref{roosW} is injective.
\smallskip

For the surjectivity, take $\alpha$ in $\Homs(\Ba^n_X,\W^{n,*}_X)$.
If $h'$ is nonvanishing at a point on $X_{reg}$, then it generates $\Ba^n_X$ and
thus $\alpha$ is determined by $\phi := \alpha h' $ there. By \cite[Theorem~3.7]{AW3}, $\phi = \psi \wedge h'$ for a unique current $\psi$ in $\W^{0,*}_X$ so  by $\Ok_X$-linearity $\alpha h=\psi\w h$ for any $h$. Hence, $\psi$ is well-defined as a current in $\W^{0,*}_X$ on $X_{\rm reg}$.

We must verify that $\psi$ has an extension in $\W_X^{0,*}$ across $X_{sing}$.
Since such an extension must be unique by the SEP, the statement is local on $X$.
Thus we may assume that $\alpha$ is defined on the whole of $X$ and that
there is a generically nonvanishing holomorphic $n$-form
$\gamma$ on $X$. Then  $\alpha \gamma$ is a section of $\W^{n,*}(X)$.

Let us choose a smooth modification $\pi\colon X'\to X$ that is biholomorphic outside
$X_{sing}$. Then $\pi^*\gamma$ is a holomorphic $n$-form on $X'$ that is generically non-vanishing.
    We claim that there is a current $\tau$ in $ \W^{n,0}(X')$ such that $\pi_*\tau=\alpha \gamma$.
    In fact, $\tau$ exists on $\pi^{-1}(X_{reg})$ since $\pi$ is a biholomorphism there.
    Moreover, by \cite[Proposition~1.2]{Apm}, $\alpha h$ is the direct image of some pseudomeromorphic
    current $\tilde \tau$ on $X'$, and is therefore also the image of the (unique) current   $\tau=\1_{\pi^{-1}(X_{reg})}\tilde\tau$ in $\W^{n,*}(X')$.

 By \cite[Theorem~3.7]{AW3} again $\tau$ is locally
    of the form $\xi\w ds$, where $\xi$ is in $\W^{0,*}_{X'}$ and $ds=ds_1\w\cdots\w ds_n$ for
    some local coordinates $s$. Hence, $\tau$ is a $K_{X'}$-valued section of $\W^{0,*}(X')$,
    so $\tau/\pi^*\gamma$ is a section of $\W^{0,*}(X')$.
    Now  $\Psi:=\pi_*(\tau/\pi^*\gamma)$ is a section of $\W^{0,*}(X)$.
On $X_{reg} \cap \{ \gamma\neq 0 \}$ we thus have that
$\Psi\w\gamma=\pi_*\tau=\alpha \gamma=\psi\w\gamma$ and so
$\Psi=\psi$ there. By the SEP it follows that $\Psi$ coincides with $\psi$ on
$X_{reg}$ and is thus the desired \pmm extension to $X$.
\end{proof}

In view of \eqref{hastverk} and \eqref{hastverk2} we have, given a local
embedding $i\colon X\to\Omega$,   the extrinsic representation
\begin{equation}\label{hastverk3}
\W_X^{0,*}\simeq \Homs(\Homs(\Ok_\Omega/\J,\CH_\Omega^Z),
\Homs(\Ok_\Omega/\J,\W_\Omega^Z)), \ \
\phi\mapsto (i_* h\mapsto i_*(\phi\w h)).
\end{equation}

\begin{lma} \label{W0Xreg} Assume that $X_{reg}\to \Omega$ is a local embedding
and $(z,w)$ coordinates as before. Each section $\phi$ in $\W^{0,*}_X$ has a unique representation \eqref{eq:W0Xreg}
    with $\widehat{\phi}_j$ in $\W^{0,*}_Z$.
\end{lma}

A current with a representation \eqref{eq:W0Xreg} is considered as an
element of $\W^{0,*}_X = \Homs(\Ba^n_X,\W^{n,*}_X)$ in view of the
comment after  Lemma~\ref{motor}.

\begin{proof}
    From \eqref{parra} we get an induced sequence
    \begin{equation}\label{parra2}
        0\to (\W_Z^{0,*})^\nu\xrightarrow{T}(\W_Z^{0,*})^M\xrightarrow{A}(\W_Z^{0,*})^{M'},
    \end{equation}
    which is also exact. In fact, $T$
    in \eqref{parra2} is clearly injective, and by \eqref{Thomotopy},
    if $\xi$ in $(\W_Z^{0,*})^M$ and  $A \xi = 0$,
    then $T \eta = \xi$, if $\eta = S \xi$.

 Now take $\phi$ in $\Homs(\Ba^n_X,\W^{n,*}_X)$.  Let us choose a basis
$\mu^1,\dots,\mu^m$ for  $\Ba^n_{X}$ and let $\tilde\phi$ be the element
in $ (\W_Z^{0,*})^M$ obtained from the coefficients of $\phi \mu^j$ when
expressed as in \eqref{paraply}, cf.\ Section \ref{gurka}.
We claim that $A\tilde{\phi} = 0$. Taking this for granted,
by the exactness of \eqref{parra2},  $\tilde\phi$ is the image of
the tuple  $\hat\phi=S\tilde\phi$.  Now $\hat\phi\w\mu^j =\phi\mu^j$ since they are represented by the same tuple in $ (\W_Z^{0,*})^M$.  Thus $\hat\phi$ gives the desired representation of $\phi$.

\smallskip
In view of Proposition~\ref{Wsmooth} it is enough to prove the claim
where $\tilde{\phi}$ is smooth.
  Let us therefore fix such a point, say $0$, and show that $(A\tilde{\phi})(0) = 0$.
    From the proof of Lemma~\ref{pointwise}, if we let $\mathcal{I}$ be the ideal
    generated by $z$, and let $X_0$ be defined by $\Ok_{X_0} := \Ok_\Omega/(\J+\mathcal{I})$, then $\mu^1\wedge \mu^z,\dots,\mu^m \w \mu^z$
    generate $\Ba^0_{X_0}$.
    If we let $\phi_0$ be the morphism in $\Homs(\Ba^0_{X_0},\Ba^0_{X_0})$ given by
    $\phi_0(\mu^i \wedge \mu^z) := \phi \mu^i \wedge \mu^z$ (which indeed gives a
    well-defined such morphism), then, as in the proof of Lemma~\ref{pointwise},
    $\tilde{\phi}_0 = \tilde{\phi}(0)$.
    In addition, the sequence \eqref{parra} for $X_0$ is
    \begin{equation*}
        0 \to \C^\nu \stackrel{T(0)}{\to} \C^M \stackrel{A(0)}{\to} \C^{M'}.
    \end{equation*}
    Since $X_0$ is $0$-dimensional, the morphism $\Ok_{X_0} \to \Homs(\Ba_{X_0},\Ba_{X_0})$
    is an isomorphism by Theorem~\ref{roossats}, and thus $\phi_0$ is given as multiplication by
    a function in $\Ok_{X_0}$, which we also denote by $\phi_0$, i.e.,
    $\tilde{\phi}_0 = T(0) \hat{\phi}_0$.
    Hence, $A(0) \tilde{\phi}_0 = A(0) T(0) \hat{\phi}_0 = 0$, and thus
$(A \tilde{\phi})(0) = 0$.
\end{proof}

\begin{ex}[Meromorphic functions]
Assume that we have a local embedding $X\to \Omega$. Given meromorphic functions
$\Phi,\Phi'$ in $\Omega$  that are holomorphic generically on $Z$, we say that
$\Phi\sim\Phi'$ if and only if $\Phi-\Phi'$ is in $\J$ generically on $Z$.
If $\Phi=A/B$ and $\Phi'=A'/B'$, where $B$ and $B'$ are generically non-vanishing on $Z$,
the condition is precisely that $AB'-A'B$ is in $\J$.  We say that such an
equivalence class is a meromorphic function $\phi$ on $X$, i.e., $\phi$ is in $\M_X$.  Clearly we have
$\Ok_X\subset\M_X.
$
We claim that
$$
\M_X\subset\W_X^{0,*}.
$$
To see this, first notice that if we take a representative $\Phi$ in $\M_\Omega$
of $\phi$, then it can be considered as an almost semi-meromorphic current on $\Omega$
with Zariski-singular support of positive codimension on $Z$, since it is generically
holomorphic on $Z$. As in Definition \ref{ainbusk} we therefore have a
current $\Phi\w h$ in $\W_X^{n,0}$ for $h$ in $\Ba^n_X$. Another representative
$\Phi'$ of $\phi$ will give rise to the same current generically and hence everywhere
by the SEP.  Thus
$\phi$ defines a section of $\Homs(\Ba^n_X,\W^{n,*}_X) = \W^{0,*}_X$.
\end{ex}


By definition, a current $\phi$ in $\W^{0,*}_X$ can be multiplied by a current $h$ in $\Ba^n_X$,
and the product $\phi \wedge h$ lies in $\W^{n,*}_X$. It will be crucial that we can extend to
products by somewhat more general currents.
Notice that $\Ba^n_X$ is a subsheaf of $\Cu^{n,*}_X$, which is an $\E^{0,*}_X$-module.
Thus, we can consider the subsheaf $\E^{0,*}_X \Ba^n_X$ of $\Cu^{n,*}_X$
which consists of finite sums $\sum \xi_i \wedge h_i$, where $\xi_i$ are in $\E^{0,*}_X$
and $h_i$ are in $\Ba^n_X$.

\begin{lma} \label{lma:extendW0}
    Each $\phi$ in $\W^{0,*}_X = \Homs_{\Ok_X}(\Ba^n_X,\W^{n,*}_X)$
has a    unique extension to a morphism in $\Homs_{\E^{0,*}_X}(\E^{0,*}_X \Ba^n_X,\W^{n,*}_X)$.
\end{lma}

\begin{proof}
    The uniqueness follows by $\E^{0,*}_X$-linearity, i.e., if $b = \xi_1 \wedge h_1 + \dots + \xi_r \wedge h_r$ is in
    $\E^{0,*}_X \Ba^n_X$, then one must have
    \begin{equation} \label{smoothtimesomega}
    \phi b= \sum_i (-1)^{(\deg \xi_i)(\deg \phi)} \xi_i \wedge \phi h_i.
    \end{equation}
    We must check that this is well-defined, i.e., that the right hand side does not
depend on the representation $\xi_1 \wedge h_1 + \dots + \xi_r \wedge h_r$ of $b$.
    By the SEP, it is enough to prove this locally on $X_{\rm reg}$,
    and we can then assume that $\phi$ has a representation \eqref{eq:W0Xreg}.
    By Proposition~\ref{Wsmooth}, it is then enough to prove that it is well-defined
    assuming that $\widehat{\phi}_0,\dots,\widehat{\phi}_{\nu-1}$ in \eqref{eq:W0Xreg}
    are all smooth. In this case, the right hand side of \eqref{smoothtimesomega}
    is simply the product of $\xi_1 \wedge h_1 + \dots + \xi_r \wedge h_r = b$
    by the smooth form $\phi$ in $\E^{0,*}_X$, and this product only depends
    on $b$. 
\end{proof}

\begin{cor} \label{cor:extendW0}
    Let $\phi$ be a current in $\W^{0,*}_X$ and let $\alpha$ be a current
    in $\W^{n,*}_X$ of the form $\alpha = \sum a_i \wedge h_i$, where
    $a_i$ are almost semi-meromorphic $(0,*)$-currents on $\Omega$ which are generically smooth on $Z$,
    and $h_i$ are in $\Ba^n_X$. Then one has a well-defined product
    \begin{equation} \label{eq:W0asmomega}
        \phi \wedge \alpha = \sum (-1)^{(\deg a_i) (\deg \phi)} a_i \wedge (\phi \wedge h_i).
    \end{equation}
\end{cor}

\begin{proof}
    The right hand side of \eqref{eq:W0asmomega} exists as a current in $\W^{n,*}_X$,
and we must prove is that it only depends on the current $\alpha$ and not on the representation  $\sum a_i \wedge h_i$.
    Notice that all the $a_i$ are smooth outside some subvariety $V$ of $Z$ and there
the right hand side of \eqref{eq:W0asmomega} is the
    product of $\phi$ and  $\alpha$ in $\E^{0,*}_X \Ba^n_X$, cf.\ Lemma~\ref{lma:extendW0}.
   It follows by the SEP that the right hand side only depends on $\alpha$.
\end{proof}


\begin{remark}\label{potta}
    Recall  from \eqref{omegatheta} that $\omega = b \vartheta$.
If $\phi$ is in $\W^{0,*}_X$, then we can define the product $\phi \w \omega$
by Corollary \ref{cor:extendW0}.
Expressed extrinsically, if $\mu = i_* \vartheta$,
   and  if we write $R\w dz = b \mu$ as in Lemma~\ref{bmy}, then
    we  can define the product $R\w dz\w\phi:= b\mu\w\phi$ as a current in $\Homs(\Ok_\Omega/\J,\W_\Omega^Z)$.
\end{remark}

\begin{lma} \label{timesOmega}
    Assume that $\phi$ is in $\W^{0,*}_X$, and that $\phi \w \omega = 0$
    for some structure form $\omega$, where the product is defined by Remark~\ref{potta}.
    Then $\phi = 0$.
\end{lma}

\begin{proof}
    Considering the component with values in $E_p$, we get that $\phi \w \omega_0 = 0$.
    By Proposition~\ref{takyta}, any $h$ in $\Ba^n_X$ can be written as
    $h = \xi \omega_0$, where $\xi$ is a holomorphic section of $E_p^*$,
    so by $\Ok$-linearity, $\phi \w h = 0$, i.e., $\phi = 0$.
\end{proof}

\medskip

We end this section with the following result, the first part of \cite[Theorem 3.7]{AW3}.
We include here a different proof than the one in \cite{AW3},
since we believe the proof here is instructive.

\begin{prop} \label{stek}
If $Z$ is smooth, then $\W_Z$ is closed under holomorphic differential operators.
\end{prop}

\begin{proof} Let $\tau$ be any current in $\W_Z$. It suffices to prove that if $\zeta$
    are local coordinates on $Z$, then $\partial \tau/\partial \zeta_1$ is in $\W_Z$.
    Consider the current
$$
\tau'=\tau\otimes \dbar\frac{dw}{2\pi i w^2}
$$
on the  manifold $Y := Z \times \C_w$. Clearly $\tau'$ has support on $Z$, and it follows from
\eqref{bost} that $\tau'$ is in $\W^Z_Y$. Let
$$p : (z,w) \mapsto \zeta = (z_1+w,z_2,\dots,z_n),$$
which is just a change of variables on $Y$ followed by a projection.
It follows from \eqref{stare3}
that $p_* \tau'$ is in $\W_Z$. Since
$$
\dbar\frac{dw}{2\pi i w^2}.\xi(w)=\frac{\partial\xi}{\partial w}(0)
$$
it is readily verified
that $p_*\tau'=\partial\tau/\partial \zeta_1$, so we conclude that $\partial \tau/\partial \zeta_1$ is in $\W_Z$.
\end{proof}

\section{The $\dbar$-operator on $\W_X^{0,*}$} \label{sect:dbarW0}

We already know the meaning of $\dbar$ on $\W_X^{n,*}$, and we now define
$\dbar$ on $\W_X^{0,*}$.
\begin{df}
Assume that $\phi,v$ are in $\W_X^{0,*}$, We say that $\dbar v=\phi$ if
\begin{equation}\label{dstreck}
\dbar (v \w h)=\phi\w h,\quad  h\in \Ba_X^n.
\end{equation}
\end{df}

If we have an embedding $X\to \Omega$,   \eqref{dstreck} means,
cf.\ \eqref{hastverk3},
 that
\begin{equation}\label{dstreck2}
    \dbar(v\w\mu)=\phi\w\mu,\quad \mu\in\Homs(\Ok_\Omega/\J,\CH_\Omega^Z).
\end{equation}

In view of Remark~\ref{potta} we can define the product $\phi \wedge \omega$
for $\phi$  in $\W^{0,*}_X$.

\begin{df}
We say that $v$ belongs to
$\Dom \dbar_X$ if  $v$ is in $\Dom\dbar$, i.e., $\dbar v=\phi$ for some $\phi$  and in addition
$\dbar(v\w\omega)$, a~priori only in $\PM_X^{n,*}$, is in $\W_X^{n,*}$, for each structure form $\omega$ from any possible embedding.
\end{df}


If $X$ is Cohen-Macaulay, then
any  such $\omega$ is of the form $a_1 h^1+\cdots +a_m h^m$,
where  $h^j$ are in $\Ba_X^n$ and $a_j$ are smooth, see
Remark~\ref{mars}, and hence $\Dom \dbar_X$ coincides with
$\Dom \dbar$ in this case.

\begin{ex} \label{plastpase}
Assume that $v$ is in $\E_X^{0,*}$ and $\phi=\dbar v$ in the sense in Section~\ref{boxer}.
Then clearly
$$
\dbar(v\wedge \omega) = \phi \wedge \omega  + (-1)^{\deg v} v \wedge \dbar \omega.
$$
Since $\dbar \omega = f \omega$, and $\W_X^{n,*}$ is closed under multiplication with
forms in $\E_X^{0,*}$, we get that $\dbar(v\wedge \omega)$ is in $\W^{n,*}_X$,
so $v$ is in $\Dom \dbar_X$ and $\dbar_Xv=\phi$.

If  $w$ is in $\Dom \dbar_X$ and $v$ is in $\E_X^{0,*}$, then
\begin{equation*}
    \dbar( v \wedge w \wedge \omega) = \dbar v \wedge w \wedge \omega + (-1)^{\deg v} v \wedge \dbar (w \wedge \omega).
\end{equation*}
Thus $v \wedge w$ is in $\Dom \dbar_X$, and the Leibniz rule
$\dbar(v \wedge w) = \dbar v \wedge w + (-1)^{\deg v} v \wedge \dbar w$ holds.
\end{ex}

Let $\chi_\delta=\chi(|h|^2/\delta)$ where $h$ is a tuple of holomorphic functions
that cuts out $X_{sing}$. 

\begin{lma} \label{dbarchi}
    If $v$ is in $\W^{0,*}(X)$, and it is in $\Dom \dbar_X$ on $X_{\rm reg}$, then
    $v$ is in $\Dom \dbar_X$ on all of $X$ if and only if
    \begin{equation}\label{bcond}
        \dbar\chi_\delta\w v\w\omega\to 0, \quad \delta\to 0,
    \end{equation}
    for all structure forms $\omega$.
    In this case,
    \begin{equation} \label{nablaOmega}
        -\nabla_f(v\w \omega) = \dbar v \w \omega.
    \end{equation}
\end{lma}

\begin{proof}
    Since $\W^{n,*}_X$ is closed under multiplication by $f$,
    $v$ is in $\Dom \dbar_X$ if and only if $\nabla_f( v \wedge \omega)$ is in  $\W^{n,*}_X$ for all structure forms $\omega$.
    Since $v$ is in $\Dom \dbar_X$ on $X_{\rm reg}$, thus  $\nabla_f( v \w \omega)$ is in $\W^{n,*}_X$    on $X_{\rm reg}$.
By \eqref{stare}, $\nabla_f(v \w \omega)$ is then in $\W^{n,*}_X$ on all of $X$
    if and only if
    \begin{equation} \label{SEPonXsing}
        {\bf 1}_{X_{\rm reg}} \nabla_f(v \wedge \omega) = \nabla_f(v \wedge \omega).
    \end{equation}
    By the Leibniz rule,
    \begin{equation} \label{nablaChi}
        \nabla_f(\chi_\delta v \w \omega) = -\dbar \chi_\delta \w v \w \omega + \chi_\delta \nabla_f (v \w \omega).
    \end{equation}
    Since $v$ is in $\W^{0,*}_X$, $v \w \omega$ is in $\W^{n,*}_X$, so the left hand side of \eqref{nablaChi} tends to
    $\nabla_f(v \w \omega)$ when $\delta \to 0$, whereas
the second term on the right hand side of \eqref{nablaChi} tends to
    ${\bf 1}_{X_{\rm reg}}\nabla_f( v \w \omega)$.
Thus \eqref{SEPonXsing} holds if and only if \eqref{bcond} does. Thus the first
statement in the lemma is proved.

    Recall, cf.\ \eqref{omegatheta}, that $\omega = b \vartheta$  where $b$ is smooth
    on $X_{\rm reg}$ and $\vartheta$ is in $\Ba^n_X$. By the Leibniz rule
thus $-\nabla_f( v \w \omega) = \dbar v \w \omega$ on $X_{\rm reg}$, since
$\nabla_f \omega = 0$.
      Therefore,   \eqref{nablaChi} is equivalent to
    $ -\nabla_f(\chi_\delta v \w \omega) = \dbar \chi_\delta \w v \w \omega + \chi_\delta \dbar v \w \omega.
    $
    If \eqref{bcond} holds, we therefore get \eqref{nablaOmega} when
$\delta \to 0$.
\end{proof}

\begin{remark} In case  $X$ is reduced the definition of $\dbar_X$ is precisely the
same as in \cite{AS}.
However, the definition of $\dbar v=\phi$ given here, for $v,\phi$ in $\W_X^{0,*}$,  does {\it not} coincide with the definition 
in, e.g., \cite{AS}.
In fact, that definition means that $\dbar (v \w h)=\phi\w h$ for
all {\it smooth} $h$ in $\Ba_X^n$, which in general is a strictly weaker condition.
For example, for any weakly holomorphic function $v$, we have $\dbar(v \w h) = 0$
for all smooth $h$ in $\Ba_X^n$, while if $X$ is a reduced complete intersection, or more generally
Cohen-Macaulay, then $\dbar(v \wedge h) = 0$ for all $h$ in $\Ba_X^n$ is equivalent to $v$ being strongly holomorphic, see \cite[p. 124]{Ts} and \cite{Astrong}.
\end{remark}

We conclude this section with a lemma that shows  that $\dbar$  means what one should expect when $\phi,v$ are expressed with respect to a local
basis $w^{\alpha_j}$ for $\Ok_X$ over $\Ok_Z$
as in  Lemma~\ref{W0Xreg}.

\begin{lma} Assume that we have a local embedding $X_{reg}\to \Omega$
and $\phi,v$ in $\W_X^{0,*}$ represented as in \eqref{eq:W0Xreg}. Then
  $\dbar v=\phi$ if and only if
\begin{equation}\label{jarl}
\dbar \hat v_j=\hat\phi_j, \quad j=0,\ldots,\nu-1.
\end{equation}
\end{lma}

\begin{proof} Let us use the notation from the proof of  Lemma~\ref{W0Xreg}.
Recall that $\hat v=S\tilde v$.
In view of \eqref{dstreck2}  and \eqref{parsol1}, $\widetilde{\dbar v}=\dbar\tilde v$.
Since $S$ is holomorphic therefore
$
\widehat{\dbar v} = S \widetilde{\dbar v} = S \dbar\tilde{v} = \dbar(S\tilde{v}) = \dbar\hat{v}.
$
\end{proof}

\section{Solving  $\dbar  u=\phi$ on $X$} \label{sect:solveDbar}

We will find local solutions to the $\dbar$-equation on $X$ by means of integral formulas.
We use the notation and machinery from \cite[Section~5]{AS}. Let $i\colon X\to\Omega\subset\C^N$
be a local embedding such that $\Omega$ is pseudoconvex, let $\Omega' \subset\subset \Omega$ be a relatively
compact subdomain of $\Omega$, and let $X' = X \cap \Omega'$.

\begin{thm}\label{klas0}
There are integral operators
$$
K\colon\E^{0,*+1}(X)\to\W^{0,*}(X')\cap \Dom \dbar_X, \quad
P\colon\E^{0,*}(X)\to\E^{0,*}(X')
$$
such that, for $\phi \in \E^{0,k}(X)$,
\begin{equation} \label{koppelman-statement}
    \phi = \dbar K\phi +K(\dbar\phi) + P\phi.
\end{equation}
\end{thm}
The operators $K$ and $P$ are described below; they depend on a choice of  weight $g$.
Since $\Omega$ is Stein one can find such a weight $g$ that is holomorphic in $z$,
by which we mean that it depends holomorphically on $z\in\Omega'$
and has no components containing any $d\bar{z}_i$, cf.\ Example~5.1 in \cite{AS}.
In this case, $P\phi$ is holomorphic when $k=0$,
and vanishes when $k\ge 1$, i.e.,
\begin{equation}
    \phi = \dbar K\phi + K(\dbar \phi), \quad \phi \in \E^{0,k}(X), \quad k \geq 1.
\end{equation}
If $\dbar\phi=0$ in $\Omega$, and $k\ge 1$, then $K\phi$ is a solution to
$\dbar v=\phi$.  If $k=0$, then $\phi=P\phi$ is holomorphic. It follows that
a smooth $\dbar$-closed function is holomorphic. In the reduced case this is
a classical theorem of Malgrange, \cite{Mal}.
In  Section~\ref{pudding} we prove that $K\phi$ is smooth on $X_{reg}$.

\smallskip

We now turn to the definition of $K$ and $P$. For future need, in Section~\ref{fine},
we define them acting on currents in $\W^{0,*}(X)$ and not only on smooth forms.
Let $\pi : \Omega_\zeta \times \Omega_z' \to \Omega_z'$ be the natural projection.
Let us choose a holomorphic Hefer form\footnote{We are only concerned with  the
component $H^0$ of this form, so for simplicity we write just $H$.}
$H$, a smooth weight $g$ with
compact support in $\Omega$ with respect to
$z\in\Omega'\subset\subset\Omega$,
and let $B$ be the Bochner-Martinelli form.
Since we are only are concerned with $(0,*)$-forms, we will here assume that $H$ and $B$ only
have holomorphic differentials in $\zeta$, i.e., the factors
$d\eta_i = d\zeta_i - dz_i$ in $H$ and $B$ in \cite{AS} should be replaced by just $d\zeta_i$.

If $\gamma$ is a current in $\Omega_\zeta\times\Omega'_z$
we let $(\gamma)_N$ be the component of bidegree $(N,*)$ in $\zeta$
and $(0,*)$ in $z$,
and let $\vartheta(\gamma)$ be the current such that
\begin{equation} \label{nudef}
\vartheta( \gamma) \w d\zeta = (\gamma)_N.
\end{equation}

Consider now $\mu$ in $\Homs(\Ok_\Omega/\J,\W^Z_\Omega)$ and $\phi$ in $\W^{0,*}_X$.
We can give meaning to
\begin{equation} \label{P0}
    (g \w HR(\zeta))_N\w\phi(\zeta) \w \mu(z)
\end{equation}
as a tensor product of currents in the following way:
First of all, by Remark~\ref{potta}, we can form the product $R(\zeta) \w d\zeta \w \phi(\zeta)$
as a current in $\W^Z_\Omega$. In view of \cite[Corollary~4.7]{AW4}
the tensor product $R(\zeta) \w d\zeta \w \phi(\zeta) \w \mu(z)$ is in
$\W^{Z \times Z'}_{\Omega_\zeta \times \Omega_z'}$, where $Z' = Z \cap \Omega'$.
Finally, we multiply this with the smooth form
$\vartheta(g \w H)$ to obtain \eqref{P0}.
Similarly, outside of $\Delta$, the diagonal in $\Omega\times\Omega'$, where $B$ is smooth, we can define
\begin{equation} \label{K0}
      (B\w g \w HR(\zeta))_N\w\phi(\zeta) \w \mu(z)
\end{equation}
as a tensor product of currents.

\begin{lma} \label{KPlemma}
   For $\mu$ in $\Homs(\Ok_{\Omega'}/\J,\W^{Z'}_{\Omega'})$
   and $\phi \in \W^{0,*}(X)$,
   the current \eqref{K0}, a priori defined as a current in $\W^{Z \times Z'\setminus \Delta}_{\Omega_\zeta \times \Omega_z'\setminus \Delta}$
   has an extension across $\Delta$.
   The current \eqref{P0} and the extension of \eqref{K0} depend $\Ok_\Omega/\J$-bilinearly on $\mu$ and $\phi$,
   and are such that
\begin{equation} \label{Kphidef}
 K\phi\w \mu := \pi_*\big( (B\w g \w HR(\zeta))_N\w\phi(\zeta) \w \mu(z)\big)
\end{equation}
and
\begin{equation} \label{Pphidef}
 P\phi\w \mu := \pi_*\big( (g \w HR(\zeta))_N\w\phi(\zeta) \w \mu(z) \big)
\end{equation}
are in  $\Homs(\Ok_{\Omega'}/\J,\W^{Z'}_{\Omega'})$.
\end{lma}

It follows that $K\phi\w\mu$ and $P\phi\w\mu$ are $\C$-linear in $\phi$
and $\Ok_{\Omega'}/\J$-linear in $\mu$.
In view of \eqref{hastverk3}, by considering $\mu$
in $\Homs(\Ok_{\Omega'}/\J,\CH^{Z'}_{\Omega'})$,
we have defined linear operators
\begin{equation}\label{batman}
K \colon\W^{0,*+1}(X) \to \W^{0,*}(X'), \quad  P \colon \W^{0,*}(X) \to \W^{0,*}(X').
\end{equation}


\begin{proof}[Proof of Lemma~\ref{KPlemma}]
In order to define the extension of \eqref{K0} across $\Delta$, we note first that
since $B$ is almost semi-meromorphic with Zariski singular support $\Delta$,
$\vartheta(B \w g\w H)$ is an almost semi-meromorphic $(0,*)$-current
on $\Omega_\zeta \times \Omega_z'$, which is smooth outside the diagonal.
We can thus form the current $\vartheta(B \w g \w H) \w R(\zeta) \w d\zeta \w\phi(\zeta)\w \mu(z)$
in $\W^{Z\times Z'}_{\Omega_\zeta \times \Omega_z'}$, cf.\ Proposition~\ref{asmW},
and this is the extension of \eqref{K0} across $\Delta$.

From the definitions above, it is clear that \eqref{P0} and the extension of \eqref{K0} are $\Ok_\Omega$-bilinear
in $\phi$ and $\mu$.
Both these currents are annihilated by $\J_z$ and $\J_\zeta$, cf.\ \eqref{korp}, so they depend
$\Ok_\Omega/\J$-bilinearly. In view of
\eqref{stare3} we conclude that \eqref{Kphidef} and \eqref{Pphidef} are in $\Homs(\Ok_{\Omega'}/\J,\W^{Z'}_{\Omega'})$.
 \end{proof}

\begin{prop}\label{pglatt}
If $\phi\in\W^{0,k}(X)$, then  $P\phi\in \E^{0,k}(X')$,
and if in addition $g$ is holomorphic in $z$,
then $P\phi\in\Ok(X')$ if $k=0$ and vanishes if $k\ge 1$.
\end{prop}

\begin{proof}
Since $\vartheta(g\w H)$ is smooth,
we get that
\begin{multline*}
\pi_*\big(\vartheta(g\w H)\w R(\zeta) \w d\zeta\w\phi \w \mu(z)\big)=\\
\pi_*\big(\vartheta(g\w H)\w R(\zeta) \w d\zeta\w\phi\big) \w \mu(z)=
\pi_*\big((g\w HR)_N\w\phi\big) \w \mu(z),
\end{multline*}
cf.\ for example \cite[(5.1.2)]{Ho1}.
Thus
$P\phi(z) = \pi_* \big( (g \w HR(\zeta))_N \w \phi \big)$
which is smooth on $\Omega'$.  If $g$ depends holomorphically on
$z$, then $P\phi$ is holomorphic in $\Omega'$ if $\phi$ is a $(0,0)$-current, and
vanishes for degree reasons if $\phi$ has positive degree.
\end{proof}

We shall now see that we can approximate $K\phi$ by smooth forms.
Let   $B^\epsilon = \chi(|\zeta-z|^2/\epsilon) B$.

\begin{prop} \label{smooth}
For any $\phi\in\W^{0,k}(X)$, $k\ge 1$,
  \begin{equation*}
      K^\epsilon \phi := \pi_* \big( (B^\epsilon\w g \w HR(\zeta))_N\w\phi\big)
=\pi_* \big( \vartheta(B^\epsilon\w g\w H)\w R(\zeta) \w d\zeta \w\phi\big)
    \end{equation*}
   is in $\E^{0,k-1}(X')$ and $K^\epsilon\phi\to K\phi$ when
$\epsilon\to 0$.
\end{prop}

The last statement means that
\begin{equation}\label{bula}
K^\epsilon\phi\w\mu\to K\phi\w\mu, \quad \mu\in
\Homs(\Ok_{\Omega'}/\J, \CH_{\Omega'}^{Z'}).
\end{equation}

\begin{proof}
Since $B^\epsilon$ is smooth, the current we push forward is $R(\zeta)\w\phi(\zeta)$ times a smooth form of $\zeta$ and $z$. Therefore
$K^\epsilon\phi$ is smooth. As in the proof of Proposition~\ref{pglatt}, we obtain since $B^\epsilon$ is smooth that
\begin{equation}\label{bus}
K^\epsilon\phi \w\mu=
\pi_*\big((B^\epsilon\w g \w HR(\zeta))_N \wedge \phi \w \mu(z)\big).
\end{equation}
By \eqref{asmWn} applied to $a=B$ we have that
\begin{equation}\label{bula2}
    (B^\epsilon\w g \w HR(\zeta))_N \w \phi \w \mu(z)  \to (B\w g \w HR(\zeta))_N \w \phi \w \mu(z)
\end{equation}
which implies \eqref{bula}.
 \end{proof}

\subsection{Proof of Theorem~\ref{klas0}}
By definition $K\phi$ and $P\phi$ are currents in $\W^{0,*}(X')$ such that
\eqref{Kphidef} and \eqref{Pphidef} hold for $\mu$
in $\Homs(\Ok_{\Omega'}/\J,\CH^{Z'}_{\Omega'})$.
We claim that
\begin{equation}\label{krokodil1}
 K\phi \w R\w dz=\pi_*\big((B\w g \w HR(\zeta))_N\w\phi\w R(z)\w dz\big)
\end{equation}
and
\begin{equation}\label{krokodil2}
 P\phi \w R\w dz=\pi_*\big((g \w HR(\zeta))_N\w\phi \w R(z)\w dz\big);
\end{equation}
here the left hand sides are defined in view of  Remark~\ref{potta}, whereas the
right hand sides have meaning by Lemma~\ref{KPlemma} and the fact that
$R(z) \w dz$ is in $\Homs(\Ok_{\Omega'}/\J,\W^{Z'}_{\Omega'})$ by Corollary~\ref{rattmuff}.

Recall from Lemma \ref{bmy} that $R\w dz=b\mu$, where $\mu$ is a tuple
of currents in $\Homs(\Ok_{\Omega'}/\J,\CH^{Z'}_{\Omega'})$ and $b$ is an almost
semi-meromorphic matrix that is smooth generically on $Z'$.
Therefore \eqref{krokodil1} and \eqref{krokodil2} hold where $b$ is smooth,
in view of Lemma~\ref{lma:extendW0},
and since both sides are in $\Homs(\Ok_{\Omega'}/\J,\W^{Z'}_{\Omega'})$, the
equalities hold everywhere by the SEP.

\smallskip

As in \cite{AS} we let  $R^\lambda=\dbar|f|^{2\lambda}\w U$ for $\re\lambda \gg 0$. It has an analytic continuation
to $\lambda=0$ and $R=R^\lambda|_{\lambda=0}$. Notice that
$R(z)\w B$ is well-defined since it is a tensor product
with respect to the coordinates $z, \eta=\zeta-z$. Also $R(z)\w R^\lambda(\zeta)\w B$
 admits such an analytic continuation and defines a \pmm
current\footnote{One can consider this current as
$R(z)\w B$  multiplied by the residue of the almost semi-meromorphic
current $U$ in \eqref{potatis}, cf.\  \cite[Section~4.4]{AW3}.}  when $\lambda=0$.
Let $B_{k,k-1}$ be the component of $B$ of bidegree $(k,k-1)$.

\begin{lma}\label{plutt}
    For all $k$,
\begin{equation}\label{plex}
B_{k,k-1} \w HR^\lambda(\zeta) \w R(z) |_{\lambda=0}= B_{k,k-1} \w HR(\zeta)\w R(z).
\end{equation}
\end{lma}

\begin{proof}[Proof of Lemma \ref{plutt}]
Notice that the equality holds outside $\Delta$. Let $T$ be the left hand side of \eqref{plex}.
In view of  Proposition~\ref{ball} it is therefore enough to check that $\1_\Delta T=0$. Fix $j,k$ and let
$$
T_{\ell}=B_{k,k-1} \w HR_j^\lambda(\zeta)\w R_\ell(z)|_{\lambda=0}.
$$
Clearly $T_\ell=0$ if $\ell< p$ so first assume that $\ell=p$.
Since $HR_j$ has bidegree $(j,j)$ in $\zeta$, the current vanishes
unless $j+k\le N$. Thus the total antiholomorphic degree is $\le N-n + N-1$.
On the other hand, the current has support on
$\Delta\cap Z\times Z\simeq Z\times\{pt\}$ which has codimension $N+N-n$. Thus it vanishes by the dimension principle.

We now prove by induction over $\ell\ge p$ that $\1_\Delta T_{\ell} = 0$.
Note that by \eqref{plast}, outside of $Z_\ell$, $R_\ell(z) = \alpha_\ell(z) R_{\ell-1}(z)$,
where $\alpha_\ell(z)$ is smooth.
Thus, outside of $Z_\ell \times \Omega$, $T_{\ell}$ is a smooth form times
$T_{\ell-1}$, and thus, by induction and \eqref{stare2},
$\1_\Delta T_{\ell}$ has its support in $\Delta \cap (Z_\ell \times Z) \simeq Z_\ell \times \{pt\}$,
which has codimension $\geq N + \ell+1$, see \eqref{BEsats}.
On the other hand, the total antiholomorphic degree is
 $\leq \ell+j+k-1 \leq \ell + N -1$,
so the current vanishes by the dimension principle.
We conclude that \eqref{plex} holds.
\end{proof}

By the same argument\footnote{There is a sign error in \cite[(5.2)]{AS} due to $R(z) \w dz$ being odd with respect
to the super structure. Since we here move $R(z) \w dz$ to the right, we get the correct sign.}
as for \cite[(5.2)]{AS} we have the equality
\begin{equation}\label{sommar}
\nabla_{f(z)}\big((B \w g \w HR^\lambda(\zeta))_N \w R(z)\w dz\big)=
[\Delta]'\w R(z)\w dz-(g \w HR^\lambda)_N \w R(z)\w dz ,
\end{equation}
also  for our $R$, where $[\Delta]'$ denotes the part of $[\Delta]$ where
$d\eta_i = d\zeta_i - dz_i$ has been replaced\footnote{This change is due to the fact that we do the same
change of the differentials in the definition of $H$ and $B$ above.} by $d\zeta_i$.
In view
of \eqref{plex}  we can put $\lambda=0$ in  \eqref{sommar},  and then
we get
 \begin{equation}\label{pulka0}
\nabla_{f(z)}\big((B \w g \w HR(\zeta))_N \w R(z)\w dz\big)=
[\Delta]' \w R(z)\w dz-(H R(\zeta)\w g)_N \w R(z)\w dz.
\end{equation}
Multiplying \eqref{pulka0} by the smooth form $\phi$, and using \eqref{krokodil1} and \eqref{krokodil2}, we get
$$
\phi\w R\w dz=-\nabla_f( K\phi\w R\w dz)+  K(\dbar\phi)\w R\w dz+
 P\phi\w R\w dz,
$$
or equivalently,
\begin{equation} \label{koppelman-proof}
\phi\w\omega=-\nabla_f(K\phi\w\omega)+K(\dbar\phi)\w\omega+  P\phi\w\omega.
\end{equation}
Multiplying by suitable holomorphic $\xi_0$ in $E_p^*$ such that
$f^*_{p+1}\xi_0=0$, cf.\  Proposition~\ref{takyta}, we see that
$\phi\w h=\dbar(K\phi\w h)+K(\dbar\phi)\w h+  P\phi\w h$ for all $h$ in
$\Ba_X$.  Thus by definition \eqref{koppelman-statement} holds.

Since $\W^{0,*}_X$ is closed under multiplication by $\Ok_X$, we get that $\psi$ in $\W^{0,*}_X$ is in
$\Dom \dbar_X$ if and only if $-\nabla_f(\psi\w \omega)$ is in $\W^{n,*}_X$.
Thus, we conclude from \eqref{koppelman-proof} that $K\phi$ is in $\Dom\dbar_X$ since all the other terms but
$-\nabla_f(K\phi \w \omega)$ are in $\W^{n,*}_X$.

%

\subsection{Intrinsic interpretation of $K$ and $P$}
So far we have defined $K$ and $P$ by means of currents  in  ambient space.
We used this approach in order to avoid introducing push-forwards on a non-reduced
space.  However, we will sketch here how this can be done.
We must first define the product space $X\times X'$.
Given a local embedding  $i : X \to \Omega$ as before, we have an embedding
$(i\times i)\colon X\times X' \to \Omega\times \Omega'$ such that the structure
sheaf is $\Ok_{\Omega\times\Omega'}/(\J_X+\J_{X'})$.  One can check that
this sheaf is independent of the chosen embedding, i.e., $\Ok_{X\times X'}$ is intrinsically
defined. Thus we also have definitions of all the various sheaves on
$X\times X'$ like $\E_{X\times X'}^{0,*}$.
The projection $p\colon X\times X'\to X'$ is determined  by
$p^*\phi \colon \Ok_{X'}\to \Ok_{X\times X'}$,
which in turn is defined so that
$
p^*i^*\Phi=(i\times i)^*\pi^*\Phi
$
for $\Phi$ in $\Ok_{\Omega'}$,
where
$\pi\colon \Omega\times \Omega'\to \Omega'$ as before.
Again one can check that this definition is independent of the embedding, and
also extends to smooth $(0,*)$-forms $\phi$.
Therefore, we have the well-defined mapping
$
p_*\colon \Cu_{X\times X'}^{2n, *+n}\to \Cu_{X'}^{n,*},
$
and clearly
\begin{equation}\label{alban}
i_* p_* = \pi_* (i \times i)_*.
\end{equation}
As before we have the isomorphism
$$
(i\times i)_*\colon \W^{2n,*}_{X\times X'} \simeq
\Homs(\Ok_{\Omega\times\Omega'}/(\J_X+\J_{X'}), \W_{\Omega\times\Omega'}^{Z\times Z'}).
$$
As in the proof of Lemma \ref{KPlemma} we see that  $\pi_*$ maps a current in $\W_{\Omega\times\Omega'}^{Z\times Z'}$ annihilated by $\J_{X'}$ to a current in
$\Homs(\Ok_\Omega/\J, \W_{\Omega'}^{Z'})$.  It follows by \eqref{alban} that
$$
p_*\colon\W_{X\times X'}^{2n, *+n}\to \W_{X'}^{n,*}.
$$
Now, take  $h$ in $\Ba^n_{X'}$ and let $\mu=i_*h$.
Then, cf.\ the proof of Lemma \ref{KPlemma},
$$
 (B\w g \w HR(\zeta))_N\w\phi(\zeta) \w \mu(z)=
(i\times i)_* \big(\vartheta( B\w g \w H) \w \omega(\zeta) \w \phi(\zeta) \w h\big).
$$
Thus we can define $K\phi$ intrinsically by
\begin{equation} \label{Kintrinsic}
    K\phi \w h= p_*\left( \vartheta(B \w g\w H) \w \omega(\zeta) \w \phi(\zeta) \w h(z)\right).
\end{equation}
From above it follows that $K\phi\w  h$ is in $\W_{X'}^{n,*}$.
In the same way we can define $P\phi$ by
\begin{equation} \label{Pintrinsic}
    P\phi \w h= p_*\left(\vartheta( g\w H) \w \omega(\zeta) \w \phi(\zeta) \w h(z)\right).
\end{equation}
It is natural to write
\begin{equation*}
    K\phi(z)=\int_\zeta \vartheta(B \w g \w H) \w \omega(\zeta)\w\phi(\zeta), \quad
P\phi(z)=\int_\zeta \vartheta(g \w H)\w \omega(\zeta) \w\phi(\zeta),
\end{equation*}
although the formal meaning is given by \eqref{Kintrinsic} and \eqref{Pintrinsic}.

\section{Regularity of solutions on $X_{reg}$}\label{pudding}

We have already seen,  cf.\ Proposition \ref{pglatt}, that $P\phi$ is always a smooth form.
We shall now prove that $K$ preserves regularity on  $X_{reg}$.
More precisely,

\begin{thm}\label{klas00}
    If $\phi$ in $\W^{0,*}_X$ is smooth near a point $x \in X_{\rm reg}'$, then $K\phi$ in
Theorem~\ref{klas0}
    is smooth near $x$.
\end{thm}

Throughout this section, let us choose local coordinates
$(\zeta,\tau)$ and $(z,w)$ at $x$  corresponding to the variables $\zeta$
and $z$ in the integral formulas, so that $Z = \{ (\zeta,\tau);\ \tau = 0 \}$.

\begin{lma}\label{pjosk}
Let $B^\epsilon:=\chi(|\zeta-z|^2/\epsilon) B$,
        and assume  that $\phi$ has compact support in our coordinate neighborhood.
    Then $K\phi$ can be approximated by the smooth forms
    \begin{equation*}
        K^\epsilon \phi := \pi_* \big( (B^\epsilon\w g \w HR)_N\w\phi\big).
    \end{equation*}
\end{lma}

Notice that here we cut away the diagonal $\Delta'$  in $Z\times Z'$
times $\C_\tau\times \C_w$
in contrast to Proposition~\ref{smooth}, where we only cut away the diagonal $\Delta$
in $\Omega\times \Omega'$.

\begin{proof}
Clearly $B^\epsilon$ is smooth so that each $K^\epsilon\phi$ is smooth in a
full \nbh in $\Omega'$ of $x$.
Let $T= \mu(z,w)\w (HR(\zeta,\tau)\w B\w g)_N\w\phi$, and let
$W=\Delta'\times \C_\tau\times \C_w$.  Since $\mu(z,w)\otimes R(\zeta,\tau)$ has
support on $\{ w=\tau=0\}$,  $T=\1_{\{w=\tau=0\}}T$.  Therefore,
$\1_W T= \1_W\1_{\{w=\tau=0\}}T=0$ since $W\cap \{ w=\tau=0\}\subset \Delta$
and $\1_\Delta T=0$ by definition, cf.\ Proposition \ref{ball} (i).
Now notice that $\1_W T=0$ implies \eqref{bula2} and in turn
\eqref{bula} with our present
choice of $B^\epsilon$.
\end{proof}

We first consider a simple but nontrivial example of Theorem~\ref{klas00}.
\begin{ex}
Let $X=\C_\zeta\subset\C^2_{\zeta,\ta}$ and $\J=(\ta^{m+1})$. Then $R=\dbar(1/\ta^{m+1})$.
For an arbitrary point $(z,w)$ we can choose the Hefer form
$$
H=\frac{1}{2\pi i}\sum_{j=0}^m\ta^{m-k}w^k d\ta.
$$
From the Bochner-Martinelli form $B$ we only get a contribution from the term
$$
B_1=\frac{1}{2\pi i}\frac{(\bar\zeta-\bar z)d\zeta+(\bar\ta-\bar w)d\ta}
{|\zeta-z|^2+|\ta-w|^2}.
$$
Let $\Omega'\subset\subset\Omega$  be open balls with center at the origin, and let
$\varphi=\varphi(|\zeta|^2+|\ta|^2)$ be a smooth cutoff function with support in $\Omega$
that is $\equiv 1$ in a \nbh of $\overline{\Omega'}$. Then
we can choose a holomorphic weight  $g=\varphi+\cdots$, see, \cite[Example~5.1]{AS} with respect to
$\Omega'$, and with support in $\Omega$.
Now $1,\ta,\ldots,\ta^m$ is a set of generators for $\Ok_X$ over $\Ok_Z$.
Assume that
$$
\phi=(\hat\phi_0(\zeta)\otimes 1+\cdots +\hat\phi_m(\zeta)\otimes \ta^m)d\bar\zeta
$$
is a smooth $(0,1)$-form. We want to compute $K\phi$. We know that
\begin{equation}\label{lemur}
K\phi=a_0(z)\otimes 1+\cdots +a_m(z)\otimes w^m
\end{equation}
with $a_k(z)$ in $\W_Z^{0,0}$.
By Lemma~\ref{pjosk} and its proof,  we have
smooth $K^\epsilon \phi(z,w)$ in $\Omega'$ such that
\begin{equation}\label{pumpa}
K^\epsilon\phi \w dz \w dw \w \dbar\frac{1}{w^{m+1}}\to K\phi \w dz \w dw \w \dbar\frac{1}{w^{m+1}}.
\end{equation}
It follows that
$$
a_k(z)=\lim_{\epsilon\to 0} \frac{1}{k!}\frac{\partial^k}{\partial w^k} K^\epsilon \phi(z,w)\big|_{w=0}.
$$
Notice that
\begin{multline*}
(B\w g \w HR(\ta))_2 = B_1\w g_{0,0} \w H\w \dbar\frac{1}{\ta^{m+1}} =\\
- \varphi \dbar\frac{1}{\ta^{m+1}}\w \frac{1}{(2\pi i)^2}\sum_{\ell=0}^m\ta^{m-\ell}w^\ell d\ta\w
\frac{(\bar\zeta-\bar z)d\zeta+(\bar\ta-\bar w)d\ta}
{|\zeta-z|^2+|\ta-w|^2}= \\
- \varphi \dbar\frac{d\ta}{\ta^{m+1}}\w \frac{1}{(2\pi i)^2}\sum_{\ell=0}^m\ta^{m-\ell}w^\ell \w
\frac{(\bar\zeta-\bar z)d\zeta}
{|\zeta-z|^2+|\ta-w|^2}.
\end{multline*}
For each fixed $\epsilon>0$,  $|\zeta-z|> 0$ on $\supp \chi_\epsilon$, cf.\ Lemma \ref{pjosk},  so we have
\begin{multline}\label{kronhjort}
K^\epsilon\phi(z,w)=\\
\int_{\zeta,\ta}
 \varphi  \frac{1}{(2\pi i)^2}\sum_{\ell=0}^m \dbar\frac{d\ta}{\ta^{\ell+1}} \w w^\ell
\chi_\epsilon\frac{(\bar\zeta-\bar z)d\bar\zeta\w d\zeta}
{|\zeta-z|^2+|\ta-w|^2}\w \sum_{k=0}^m \hat\phi_k(\zeta)\otimes\ta^k.
\end{multline}
A simple computation yields  that
\begin{equation}\label{kronhjort2}
 K^\epsilon\phi(z,w)=\sum_{k=0}^m a_k^\epsilon(z)\otimes w^k+\Ok(\bar w),
\end{equation}
where
$$
a_k^\epsilon(z)=\frac{1}{2\pi i}\int_\zeta\varphi(|\zeta|^2)\chi_\epsilon\frac{\hat\phi_k(\zeta)d\bar\zeta\w d\zeta}
{\zeta-z}.
$$
Letting  $\epsilon$ tend to $0$ we get $K\phi$  as in \eqref{lemur},
where
$$
a_k(z)=\frac{1}{2\pi i}\int_\zeta\varphi(|\zeta|^2)\frac{\hat\phi_k(\zeta)d\bar\zeta\w d\zeta}
{\zeta-z}.
$$
It is well-known that these Cauchy integrals $a_k(z)$ are smooth solutions to
$\dbar v=\hat\phi_k d\bar z$ in $Z'=Z\cap\Omega'$. Thus $K\phi$ is smooth.
\end{ex}

\begin{remark}
The terms $\Ok(\bar w)$ in the expansion \eqref{kronhjort2}
of $K^\epsilon\phi(z,w)$
do {\it not} converge to smooth functions in general when $\epsilon\to 0$.
For a simple example, take $\phi=\zeta d\bar\zeta\otimes \tau^m$. Then
$K^\epsilon\phi(0,w)$ tends to
$$
w^m\int\varphi(|\zeta|^2)\frac{1}{2\pi i}\frac{|\zeta|^2d\bar\zeta\w d\zeta}{|\zeta|^2+|w|^2}
$$
which is a smooth function of $w$ plus (a constant times)
$w^m|w|^2\log|w|^2$, and thus not smooth. However, it is certainly in $C^m$. One can check that
$K\phi(z,w)=\lim_{\epsilon\to 0^+}K^\epsilon\phi(z,w)$ exists pointwise and defines a function in
at least $C^m$ and that our  solution can be computed from this limit.  In fact, by a more precise
computation we get from  \eqref{kronhjort} that
$$
K^\epsilon\phi(z,w)=\sum_{k=0}^m\int_\zeta\varphi(|\zeta|^2)\chi_\epsilon
\frac{1}{2\pi i}\frac{(\bar\zeta-\bar z)\hat\phi_k(\zeta) d\bar\zeta\w d\zeta}{|\zeta-z|^2+|w|^2}
w^k\sum_{j=0}^{m-k}\left(\frac{|w|^2}{|\zeta-z|^2+|w|^2}\right)^j.
$$
It is now clear that we can let $\epsilon\to 0$. By a simple computation we then get
$$
K\phi(z,w)=\sum_{k=0}^m C\hat\phi_k(z)\otimes w^k-
\sum_{k=0}^m\int_\zeta\varphi(|\zeta|^2)
\frac{1}{2\pi i}\frac{\hat\phi_k(\zeta)d\bar\zeta\w d\zeta}{\zeta-z}
w^k\left(\frac{|w|^2}{|\zeta-z|^2+|w|^2}\right)^{m-k+1}.
$$
Let $\psi=\varphi\hat\phi_k$. Then the $k$th term in the second sum is equal to
$$
b(z,w)=\frac{1}{2\pi i}\int_\zeta\frac{\psi(z+\zeta)d\bar\zeta\w d\zeta}{\zeta} w^k
\left(\frac{|w|^2}{|\zeta|^2+|w|^2}\right)^{m-k+1}.
$$
If we integrate outside the unit disk, then we certainly get a smooth function. Thus
it is enough to consider the integral over the disk. Moreover, if $\psi(z+\zeta)
=\Ok(|\zeta|^M)$ for a large $M$, then the integral is at least $C^m$.
By a Taylor expansion of $\psi(z+\zeta)$ at the point $z$, we are thus reduced to consider
$$
\int_{|\zeta|<1}\frac{\zeta^\alpha\bar\zeta^\beta}{\zeta}
\left(\frac{|w|^2}{|\zeta|^2+|w|^2}\right)^{m-k+1}.
$$
For symmetry reasons, they vanish, except when $\alpha=\beta+1$. Thus we are left with
$$
\int_{|\zeta|<1} |\zeta|^{2\beta}
\left(\frac{|w|^2}{|\zeta|^2+|w|^2}\right)^{m-k+1} w^k=
C w^k|w|^{2(m-k+1)}\int_0^1 \frac{s^\beta ds}{(s+|w|^2)^{m-k+1}}
$$
for non-negative integers $\beta$. The right hand side is a smooth function of
$w$ if $\beta\le m-k-1$ and a smooth function plus
$$
Cw^k|w|^{2(\beta+1)}\log|w|^2
$$
if $\beta\ge m-k$. The worst case therefore is when  $k=m$ and $\beta=0$;  then we have
$
w^m|w|^{2}\log|w|^2
$
that we encountered above.
\end{remark}

\begin{prop}\label{glatt}
Let $z,w$ be coordinates at a point $x\in X_{reg}$ such that $Z=\{w=0\}$
and $x=(0,0)$.
If $\phi$ is smooth, and has support where the local coordinates are defined,
then
$$
v^\epsilon(z,w)=\int_\zeta\chi(|\zeta-z|^2/\epsilon)(HR\w B\w g)_N\w\phi,
$$
is smooth for $\epsilon>0$, and for each multiindex $\ell$ there is a smooth form
$v_\ell$ such that
$$
\partial^\ell_w v^\epsilon|_{w=0}\to v_\ell
$$
as currents on $Z$.
\end{prop}

Taking this proposition for granted we can conclude the proof of Theorem~\ref{klas00}.

\begin{proof}[Proof of Theorem~\ref{klas00}]
If $\phi \equiv 0$ in a neighborhood of $x\in X'_{\rm reg}$, then $K\phi$ is smooth near $x$, cf.\ the proof of Proposition~\ref{smooth}.
Thus, it is sufficient to prove Theorem~\ref{klas00} assuming that $\phi$ is smooth
and has support near $x$.

Recall that given a minimal generating set  $1, w^{\alpha_1},\ldots, w^{\alpha_{\nu-1}}$,
one gets the coefficients $\hat v^\epsilon_j$ in the representation
$$
v^\epsilon=\hat v_0^\epsilon\otimes 1+\cdots +\hat v_{\nu-1}^\epsilon\otimes w^{\alpha_{\nu-1}}
$$
from
$\partial^\ell_w v^\epsilon|_{w=0}$, $|\ell|\le M$  by a holomorphic matrix, cf.,
the proof of Lemma~\ref{skottskada}.  It thus follows from
Proposition~\ref{glatt} that there are smooth $\hat v_j$ such that
$\hat v_j^\epsilon\to \hat v_j$ as currents on $Z$.
Let $v=\hat v_0\otimes 1+\cdots + \hat v_{\nu-1}\otimes w^{\alpha_{\nu-1}}$.
In view of \eqref{paraply2},  $v^\epsilon\w\mu\to v\w\mu$
for all  $\mu$ in $\Homs(\Ok_\Omega/\J,\CH^Z_\Omega)$.
From Lemma~\ref{pjosk} we conclude that
$
v\w\mu= K\phi\w\mu
$
for all such $\mu$. Thus $K\phi=v$ in $\W_X^{0,*}$ and hence $K\phi$ is smooth.
\end{proof}

\begin{proof}[Proof of Proposition~\ref{glatt}]
Assume that $X$ is embedded in $\Omega\subset\C^N_{\zeta',\tau'}$. After a suitable rotation
we can assume that $Z$ is the graph $\tau'=\psi(\zeta')$. The Bochner-Martinelli kernel
in $\Omega$ is rotation invariant, so it is
$$
B=\sigma+\sigma\w\dbar\sigma+\sigma\w(\dbar\sigma)^2+\dots,
$$
where
$$
\sigma=\frac{(\bar\zeta'-\bar z')\cdot d\zeta'+(\bar\ta'-\bar w')\cdot d\ta'}
{|\zeta'-z'|^2+|\ta'-w'|^2}.
$$
We now choose the new coordinates
$
\zeta=\zeta',\   \ta=\ta'-\psi(\zeta')
$
in  $\Omega$, so that $Z=\{(\zeta,\ta);\  \ta=0\}$.

Recall that on $X_{reg}$ we have that
$R\w dz$ is a smooth form times $\mu=(\mu_1,\ldots,\mu_m)$, where $\mu_j$ is a generating set
for $\Homs(\Ok_\Omega/\J,\CH_\Omega^Z)$.
Thus we are to compute $\partial_w^\ell|_{w=0}$ of integrals like
\begin{equation}\label{smak}
\int_{\zeta,\ta}\dbar\frac{d\ta}{\ta^{\alpha+\1}}\w B^\epsilon_k\w\phi(\zeta,z,w,\ta),
\end{equation}
where $k\le n$ and $\phi$ is smooth with compact support near $x$. It is clear that
the symbols $\bar\ta, \bar w, d\bar\ta$ can be omitted in the expression for
$$
B^\epsilon=\chi_\epsilon B=\chi(|\zeta-z|^2/\epsilon)B,
$$
since $\bar \ta$ and $d \bar \ta$ annihilate $\dbar(1/\tau^{\alpha+1})$, and
since we only take holomorphic derivatives with respect to $w$ and set $w = 0$.

Let us write $\psi(\zeta)-\psi(z)=A(\zeta,z)\eta$, where
$\eta:=\zeta-z$ is considered as a column matrix and $A$ is a holomorphic $(N-n)\times n$-matrix.
Then
$$
\sigma=\frac{\eta^*\nu}{|\zeta-z|^2+|\ta-w+\psi(\zeta)-\psi(z)|^2},
$$
where $\nu$ is the $(1,0)$-form valued column matrix
$$
\nu=d\zeta+A^*d(\tau + \psi(\zeta)).
$$
Since $\eta^*\nu$ is a $(1,0)$-form we have that
$$
B^\epsilon_k=\chi_\epsilon\frac{\eta^*\nu\w ( (d\eta^*)\nu+\eta^*\dbar\nu)^{k-1}}{(|\zeta-z|^2+|\ta-w+\psi(\zeta)-\psi(z)|^2)^k}.
$$

\begin{lma}\label{gammak}
Let
$$
\xi^i=\xi^i_1\frac{\partial}{\partial\zeta_1}+\cdots+ \xi^i_n\frac{\partial}{\partial\zeta_n}
$$
be smooth $(1,0)$-vector fields, and let $L_i=L_{\xi^i}$ be the associated Lie derivatives
for $i=1,\dots,\rho$.
Let
$$
\gamma_k:=\eta^*\nu\w ((d\eta^*)\nu+\eta^*\dbar\nu)^{k-1}.
$$
If we have a modification $\pi\colon  \tilde{W} \to \Omega \times \Omega$ such that locally $\pi^*\eta=\eta_0\eta'$,
where $\eta_0$ is a holomorphic function, then
$$
\pi^*(L_1 \cdots L_\rho \gamma_k)=\bar\eta_0^k \beta,
$$
where $\beta $ is smooth.
\end{lma}

Recall that if $a$ is a form, then $L_\xi a = d(\xi\neg a)+ \xi\neg (d a)$, and that
$L_\xi (\beta\neg a)=[\xi,\beta]\neg a+\beta\neg (L_\xi a)$ if $\beta$ is another
vector field.

\begin{proof}
Introduce a  nonsense basis $e$ and its dual $e^*$ and consider the exterior
algebra spanned by $e_j, e^*_\ell,$ and the cotangent bundle.
Let
$$
c_\ell=\eta^*e\w ((d\eta^*) e)^{\ell-1}.
$$
Notice that $\gamma_k$ is a sum of terms like
$$
(\nu e^*\neg)^\ell c_\ell\w  (\eta^*\dbar\nu)^{k-\ell}.
$$
Since $L_i c_\ell=0$ and $L_i(\eta^* b)=\eta^*L_i b$
it follows after a finite number of applications of $L_i$'s that we get
$$
(\nu_1 e^*)\neg\cdots(\nu_\ell e^*)\neg c_\ell (\eta^*b_1)\cdots(\eta^*b_{k-\ell}),
$$
where $\nu_j$ and $b_j$ are smooth. Since
$$
\pi^* c_\ell = \bar\eta_0^\ell (\eta')^* e\w (d(\eta')^* e)^{\ell-1},
$$
the lemma now follows.
\end{proof}

We note that $\eta^*(I+A^*A)\eta = |\zeta-z|^2+|\psi(\zeta)-\psi(z)|^2$.
Thus, differentiating \eqref{smak} with respect to $w$, setting $w = 0$, and evaluating the residue with respect to $\tau$
using \eqref{spurt}, we obtain a sum of integrals like
$$
\int_\zeta\chi_\epsilon\frac{(\eta^* a_1)\cdots(\eta^* a_{t+1})\w \gamma_k\w \phi}{(\eta^*(I+A^*A)\eta)^{k+t+1}},
$$
where $a_1,\dots,a_{t+1}$ are column vectors of smooth functions. We must
prove that the limit of such integrals when $\epsilon\to 0$ are smooth in $z$.

\begin{lma} Let
$$
I_\ell^{r,s}=\int \chi_\epsilon \frac{(\eta^* a_1) \cdots (\eta^*a_r)\Ok(|\eta|^{2s}) \tilde\gamma_k\w\phi}{\Phi^{k+\ell}},
$$
where $a_1,\dots,a_r$ are tuples of smooth functions, $\tilde\gamma_k=L_1 \cdots L_\rho \gamma_k$, where
$L_i =L_{\xi_i}$ are Lie derivatives with respect to smooth $(1,0)$-vector fields $\xi^i$ as above for $i=1,\dots,\rho$,
$\phi$ is a test form with support close to $z$,
and $\Phi := \eta^*(I+A^*A)\eta$.
If $r\geq1$ and $r+s \geq \ell + 1$, then we have the relation
\begin{equation}\label{rekursion}
I^{r,s}_{\ell+1}=I^{r-1,s}_{\ell}+I^{r-1,s+1}_{\ell+1}+I^{r,s-1}_{\ell}+o(1)
\end{equation}
when $\epsilon\to 0$.
\end{lma}

\begin{proof}
If
$$
\xi= a_r^t (I+A^*A)^{-t}\frac{\partial}{\partial\zeta},
$$
and $L = L_\xi$, then
using that $\Phi = \eta^t (I+A^* A)^t \bar{\eta}$, one obtains that
\begin{equation}\label{svan1}
L\Phi=\eta^* a_r +\Ok(|\eta|^2).
\end{equation}
Thus
$$
I_{\ell+1}^{r,s}=\int \chi_\epsilon (\eta^* a_1) \cdots (\eta^* a_{r-1}) \Ok(|\eta|^{2s}) \tilde\gamma_k\w\phi
L \frac{1}{\Phi^{k+\ell}}+I_{\ell+1}^{r-1,s+1}
$$
in view of \eqref{svan1}.
We now integrate by parts by $L$ in the integral.  If a derivative with respect to
$\zeta_j$ falls on some $\eta^* a_i$, we get a term $I^{r-1,s}_\ell$. If it falls on
$\Ok(|\eta|^{2s})$ we get either $\Ok(|\eta|^{2(s-1)})$ times $\eta^* b$, for some tuple $b$ of smooth functions,
and this gives rise to the term $I^{r,s-1}_{\ell}$ or $\Ok(|\eta|^{2s})$, and this gives rise to another term $I^{r-1,s}_\ell$.
If it falls on $\phi$ or $\tilde{\gamma}_k$ we get an additional term
$I^{r-1,s}_\ell$.  The only possibility left is when the derivative falls on $\chi_\epsilon=
\chi(|\eta|^2/\epsilon)$.  It remains to show that an integral of the form
$$
\int_{\zeta,z}\chi'(|\eta|^2/\epsilon)\frac{(\eta^*a_1) \cdots (\eta^*a_{r-1}) (\eta^*b)}{\epsilon}
\frac{\Ok(|\eta|^{2s}) \gamma_k\w\phi}{\Phi^{k+\ell}}
$$
tends to $0$, where the factor $\eta^* b$ comes from the derivative of $|\eta|^2$.
We now choose a resolution $\widetilde V\to \Omega\times\Omega$ such that
$\eta=\eta_0\eta'$ where $\eta'$ is non-vanishing and $\eta_0$ is
(locally) a monomial.  Notice that
$\pi^*\Phi=|\eta_0|^2\Phi'$ where $\Phi'$ is smooth and strictly positive.
In view of Lemma~\ref{gammak}
we thus obtain integrals of the form
\begin{equation} \label{chiprimbar}
\int_{\widetilde  V}\chi'(|\eta_0|^2v/\epsilon) \frac{1}{\epsilon} \frac{\bar{\eta}_0^{r+s-\ell}}{\eta_0^{k+\ell-s}} \alpha,
\end{equation}
where $v$ is smooth and strictly positive and $\alpha$ is smooth.

In order to see that the limit of \eqref{chiprimbar} tends to $0$, we note first that if we let
\begin{equation*}
    \tilde{\chi}(s) = s \chi'(s) + \chi(s),
\end{equation*}
then just as $\chi$, $\tilde{\chi}$ is also a smooth function on $[0,\infty)$ that is $0$ in a \nbh of $0$ and $1$ in a \nbh of $\infty$.
By assumption, $r+s-\ell-1 \geq 0$.
Since the principal value current $1/f^m$ acting on a test form $\beta$
can be defined as
\begin{equation*}
    \lim_{\epsilon \to 0^+} \int \chi(|f|^2 v/\epsilon) \frac{\beta}{f^m}
\end{equation*}
for any cut-off function as above, the principal value current $1/\eta_0^{k+\ell-s}$ acting on $\bar{\eta}_0^{r+s-\ell-1} \alpha$ equals
\begin{equation*}
    \lim_{\epsilon \to 0^+} \int_{\widetilde  V}\chi(|\eta_0|^2v/\epsilon) \frac{\bar{\eta}_0^{r+s-\ell-1}}{\eta_0^{k+\ell-s}} \alpha =
    \lim_{\epsilon \to 0^+} \int_{\widetilde  V}\tilde{\chi}(|\eta_0|^2v/\epsilon) \frac{\bar{\eta}_0^{r+s-\ell-1}}{\eta_0^{k+\ell-s}} \alpha.
\end{equation*}
Taking the difference between the left and right hand side, we conclude that \eqref{chiprimbar} tends to $0$
when $\epsilon \to 0$.
\end{proof}

Now we can conclude the proof of Proposition~\ref{glatt}.
From the beginning we have  $I_{\ell}^{\ell,0}$. After repeated applications of \eqref{rekursion}
we end up with
$$
I_\ell^{0,\ell}+I_{\ell-1}^{0,\ell-1}+\cdots + I_0^{0,0} +o(1).
$$
However,
any of these integrals has an integrable kernel even when $\epsilon=0$.
This means that we are back to the case in \cite[Lemma~6.2]{AS}, and so the limit integral  is
smooth in $z$.
\end{proof}

\section{A fine resolution of $\Ok_X$} \label{fine}

We first consider a generalization of  Theorem~\ref{klas0}.

\begin{lma}\label{neptun}
Assume that $\phi\in\W^{0,k}(X)\cap\E_X^{0,k}(X_{reg})\cap\Dom\dbar_X$ and that
$K\phi$ is in $\Dom\dbar_X$ (or just in $\Dom\dbar$).
Then
\eqref{koppelman-statement} holds on $X'$.
\end{lma}

\begin{proof} Let $\chi_\delta$ be functions as before that cut away $X_{sing}$. From Koppelman's formula
\eqref{koppelman-statement} for smooth forms we have
\begin{equation}\label{pojke}
\chi_\delta\phi\w h=\dbar(K (\chi_\delta\phi))\w h
+ K(\chi_\delta\dbar\phi)\w h+  P(\chi_\delta \phi)\w h  + K(\dbar\chi_\delta\w\phi)\w h,   \ \  h\in  \Ba_X^n,
\end{equation}
for $z\in X'_{reg}$. Clearly the left hand side tends to
$\phi\w h $ when $\delta \to 0$.
From Lemma~\ref{KPlemma} it follows that $K(\chi_\delta \phi) \w h\to  K\phi\w h$.
Thus the first term on the right hand side of \eqref{pojke} tends to
$\dbar (K\phi)\w h$.
In the same way the second and third terms on the right hand side
tend to $K(\dbar\phi)\w h$ and $ P\phi \w h$, respectively.
It remains to show
that the last term tends to $0$. If $z$ belongs to a fixed compact subset of $X_{\rm reg}'$,
then $B$ is smooth in \eqref{K0}  when $\zeta$ is in
$\supp \dbar \chi_\delta$ for small $\delta$.
Hence it suffices to see that
$$R(\zeta) \w d\zeta \w \dbar \chi_\delta \w \phi(\zeta) \w i_*h \to 0,$$
and since this is a tensor product of currents, it suffices to see that
$$
R(\zeta) \w d\zeta \w \dbar \chi_\delta \w \phi(\zeta) \to 0,
$$
or equivalently,
$\omega(\zeta) \w \dbar \chi_\delta \w \phi(\zeta) \to 0,$
which follows by Lemma~\ref{dbarchi} since $\phi$ is in $\Dom \dbar_X$.  We have thus proved that
$$
\chi_\delta \phi\w h =\chi_\delta\dbar(K\phi)\w h+\chi_\delta K(\dbar\phi)\w h+
\chi_\delta  P\phi\w h.
$$
The first term on the right hand side is equal to $\dbar(\chi_\delta K\phi\w h)-
\dbar\chi_\delta \w K\phi\w h$, where the latter term tends to $0$ if $K\phi$
is in $\Dom\dbar_X$ or just in $\Dom\dbar$, cf.\  Lemma~\ref{dbarchi}.  Thus we get
$$
\phi\w h= \dbar( K\phi) \w h+  K(\dbar\phi)\w h+P\phi\w h, \ \ h\in  \Ba_X^n,
$$
which precisely means that \eqref{koppelman-statement} holds.
 \end{proof}


\begin{df} We say that a $(0,q)$-current $\phi$ on an open set $\U\subset X$ is a section of
$\A_X^q$ over $\U$, $\phi\in\A^q(\U)$, if, for every $x\in\U$, the germ $\phi_x$ can be written
as a finite sum of terms
$$
\xi_\nu\w K_\nu(\cdots \xi_2\w K_2(\xi_1\w K_1(\xi_0))),
$$
where $\xi_j$ are smooth $(0,*)$-forms
and $K_j$ are integral operators with kernels $k_j(\zeta,z)$ at $x$, defined as above,
and such that $\xi_j$ has compact support in the set where $z\mapsto k_j(\zeta,z)$ is
defined.
\end{df}

Clearly $\A_X^*$ is closed under multiplication by $\xi$ in $\E_X^{0,*}$.
It follows from \eqref{batman}  that $\A_X^*$ is a subsheaf of $\W_X^{0,*}$ and
from Theorem \ref{klas00}
that $\A_X^k=\E_X^{0,*}$ on $X_{reg}$.
Clearly any operator $K$ as above maps $\A_X^{*+1}\to\A_X^*$.

\begin{lma}\label{neptun2}
If $\phi$ is in $\A_X$, then $\phi$ and $K\phi$ are in $\Dom\dbar_X$.
\end{lma}

\begin{proof} Notice that \cite[Lemma~6.4]{AS} holds in our case by verbatim
the same proof, since we have access to the dimension principle for
(tensor products of) \pmm $(n,*)$-currents, and the computation  rule \eqref{stare2},
cf.\ the comment after Definition~\ref{polestar}.
Since $\A_X^* = \E^{0,*}_X$ on $X_{\rm reg}$ it is enough by Lemma~\ref{dbarchi} to check that
$
\dbar\chi_\delta\w\omega\w\phi\to 0,
$
and this precisely follows from \cite[Lemma~6.4]{AS}.
\end{proof}

In view of Lemmas \ref{neptun} and \ref{neptun2} we have

\begin{prop} Let $K,P$ be integral operators as in Theorem~\ref{klas0}. Then
$$
K\colon\A^{k+1}(X)\to\A^k(X'),\quad  P\colon\A^k(X)\to \E^{0,k}(X'),
$$
 and
the Koppelman formula \eqref{koppelman-statement} holds.
\end{prop}

\begin{proof}[Proof of Theorem~\ref{main}]
    By definition, it is clear that $\A_X^k$ are modules over $\E^{0,k}_X$,
    and by Theorem~\ref{klas00}, $\A_X^k$ coincides with $\E^{0,k}_X$ on $X_{\rm reg}$.
    Since we have access to Koppelman formulas, precisely as in the proof of \cite[Theorem~1.2]{AS}
    we can verify that $\dbar\colon \A_X^k\to\A_X^{k+1}$.

    It remains to prove that \eqref{dolb1} is exact. We choose locally a weight $g$ that is holomorphic in $z$,
    so the term $P\phi$ vanishes if $\phi$ is in $\A_X^k$, $k\ge 1$, and is holomorphic in $z$ when $k=0$.
    Assume that $\phi$ is in $\A_X^k$ and $\dbar\phi=0$. If $k\ge 1$, then
    $\dbar K\phi=\phi$, and if $k=0$, then $\phi=P\phi$.
\end{proof}

\subsection{Global solvability}
Assume that $E\to X$ is a holomorphic vector bundle; this means that
the transition matrices have entries in $\Ok_X$.
For instance if we have a global embedding $i\colon X\to \Omega$ and a holomorphic vector bundle $F\to \Omega$, then $F$ defines a  vector bundle $i^*F\to X$.
The sheaves $\A_X^{*}(E)$ give rise to a fine
resolution of the sheaf $\Ok_X(E)$, and by standard homological algebra
we have the isomorphisms
$$
H^q(X,\Ok(E))=\frac{\Ker(\A^q(X,E)\stackrel{\dbar}{\to}\A^{q+1}(X,E))}
{\Image(\A^{q-1}(X,E)\stackrel{\dbar}{\to}\A^{q}(X,E))},   \quad q\ge 1.
$$
Thus, if $\phi\in\A^{q+1}(X,E)$, $\dbar\phi=0$,  and its canonical cohomology class vanishes, then
the equation $\dbar\psi=\phi$ has a global solution in $\A^q(X,E)$. In particular, the equation is
always solvable if $X$ is Stein.
If for instance $X$ is a pure-dimensional projective variety $i\colon X\to \Pk^N$, then
the $\dbar$-equation is
solvable, e.g.,  if $E$ is a sufficiently ample line bundle.

\section{Locally complete  intersections}\label{fullst}

Let us consider the special case when $X$ locally is a complete intersection, i.e., given
a local embedding $i\colon X\to \Omega\subset\C^N$ there  are
global sections  $f_j$ of $\Ok(d_j)\to \Pk^N$
such that
$\J=(f_1,\ldots,f_p)$, where  $p=N-n$.  In particular, $Z=\{f_1=\cdots=f_p=0\}$.
In this case $\Homs(\Ok_\Omega/\J,\CH_\Omega)$ is generated by the single current
$$
\mu=\dbar\frac{1}{f_p}\w\cdots\w\dbar\frac{1}{f_1}\w dz_1\w\cdots\w dz_N,
$$
see, e.g., \cite{Aext}.  Each
smooth $(0,q)$-form  $\phi$ in $\E_X^{0,q}$ is thus represented by a current
$\Phi\w\mu$, where $\Phi$ is smooth in a \nbh of $Z$ and $i^*\Phi=\phi$.
Moreover, $X$ is Cohen-Macaulay so  $X_{reg}$ coincides with the part of $X$ where
$Z$ is regular, and $\dbar\phi=\psi$ if and only if
$\dbar(\phi\w\mu)=\psi\w\mu$.

\smallskip
Henkin and Polyakov introduced, see \cite[Definition~1.3]{HePo2},
the notion of {\it residual currents
$\phi$ of bidegree $(0,q)$} on a locally complete intersection $X\subset\Pk^N$, and the $\dbar$-equation
$\dbar\psi=\phi$. Their currents $\phi$ correspond to our $\phi$ in $\E_X^{0,q}$ and
their $\dbar$-operator on such currents coincides with ours.

\begin{remark} In \cite{HePo3} Henkin and Polyakov consider a global reduced complete intersection
 $X\subset\Pk^N$. They prove, by a global explicit formula, that if $\phi$ is a global
$\dbar$-closed smooth $(0,q)$-form with values in $\Ok(\ell)$,
$\ell= d_1+\cdots d_p  -N-1$,  then there is
a smooth solution to $\dbar\psi=\phi$ at least on $X_{reg}$, if
$1\le q\le n-1$. When $q=n$ a necessary obstruction term occurs.
However, their meaning of
``$\dbar$-closed'' is that locally there is a representative $\Phi$ of $\phi$ and  smooth
$g_j$ such that $\dbar\Phi=g_1f_1+\cdots +g_pf_p$.  If this holds, then clearly $\dbar\phi=0$.
The converse implication is {\it not} true, see Example~\ref{spott} below. It is not clear
to us whether their formula gives a solution under the weaker assumption that $\dbar\phi=0$,
neither do we know whether their solution admits some intrinsic
extension across $X_{sing}$ as a current on $X$.
\end{remark}

\begin{ex}\label{spott}
Let $X=\{f=0\}\subset\Omega\subset\C^{n+1}$ be a reduced hypersurface, and assume that
$df\neq 0$ on $X_{reg}$, so that $\J=(f)$. Let $\phi$ be a smooth $(0,q)$-form in a \nbh of some
point  $x$ on $X$ such that $\dbar\phi=0$.  We claim that $\dbar u=\phi$ has a smooth solution
$u$ if and only if $\phi$ has a smooth representative $\Phi$ in ambient space such that $\dbar\Phi=fg$
for some smooth form $g$.
In fact, if such a $\Phi$ exists then  $0=f\dbar g$ and thus $\dbar  g=0$.
Therefore,
$g=\dbar \gamma$ for some smooth $\gamma$ (in a Stein \nbh of $x$ in ambient space)
and hence $\dbar(\Phi-f\gamma)=0$.
Thus there is a smooth $U$
such that $\dbar U=\Phi-f\gamma$; this means that $u=i^*U$ is a smooth solution to $\dbar u=\phi$.
Conversely, if  $u$ is a smooth solution, then $u=i^* U$ for some smooth $U$ in ambient space, and
thus $\Phi=\dbar U$ is a representative of $\phi$ in ambient space. Thus $\dbar\Phi=fg$ (with $g=0$).

There are examples of
hypersurfaces $X$ where there exist
smooth $\phi$  with $\dbar\phi=0$ that do not admit smooth solutions to $\dbar u=\phi$,
see, e.g., \cite[Example~1.1]{AS}. It follows that such a $\phi$ cannot have a representative
$\Phi$ in ambient space as above.
\end{ex}


\begin{thebibliography}{99}

\bibitem{ACH} \textsc{M.\ Andersson:} Uniqueness and factorization of Coleff-Herrera currents.
\textit{Ann. Fac. Sci. Toulouse Math.}, \textbf{18} (2009), no. 4, 651--661.

\bibitem{Astrong} \textsc{M.\ Andersson:} A residue criterion for strong holomorphicity.
\textit{Ark. Mat.}, \textbf{48(1)} (2010), 1--15.

\bibitem{Aext} \textsc{M.\ Andersson:} Coleff-Herrera currents, duality, and Noetherian operators.
\textit{Bull. Soc. Math. France},  \textbf{139} (2011),  535--554.

\bibitem{Apm} \textsc{M.\ Andersson:} Pseudomeromorphic currents on subvarieties.
\textit{Complex Var. Elliptic Equ.}, \textbf{61} (2016), no. 11, 1533--1540.

\bibitem{ALRSW} \textsc{M. Andersson, R. L\"ark\"ang, J. Ruppenthal, H. Samuelsson Kalm, E. Wulcan:}
Estimates for the $\dbar$-equation on canonical surfaces.
\textit{Preprint}, 2018. Available at arXiv:1804.01004 [math.CV].



\bibitem{AS} \textsc{M.\ Andersson, H.\ Samuelsson:}
 A Dolbeault-Grothendieck lemma on complex spaces via Koppelman formulas.
\textit{Invent. Math.}, \textbf{190} (2012), 261--297.

\bibitem{AW1} \textsc{M.\ Andersson, E.\ Wulcan:} Residue currents with prescribed annihilator ideals.
\textit{Ann. Sci. \'{E}cole Norm. Sup.}, \textbf{40} (2007), 985--1007.

\bibitem{AW2} \textsc{M.\ Andersson, E.\ Wulcan:} Decomposition of residue currents.
\textit{J. Reine Angew. Math.}, \textbf{638} (2010), 103--118.

\bibitem{AWsemester} \textsc{M.\ Andersson, E.\ Wulcan:} Global effective versions of the Brian\c con-Skoda-Huneke theorem.
\textit{Invent. Math}, \textbf{200} (2015), no. 2, 607--651.


\bibitem{AW3} \textsc{M.\ Andersson, E.\ Wulcan:}  Direct images of semi-meromorphic currents.
\textit{Ann. Inst. Fourier} (to appear), available at arXiv:1411.4832v2 [math.CV].


\bibitem{AW4} \textsc{M.\ Andersson, E.\ Wulcan:}  Regularity of pseudomeromorphic currents.
\textit{Preprint}, 2017. Available at arXiv:1703.01247 [math.CV].

\bibitem{Barlet} \textsc{D.\ Barlet:} Le faisceau $\omega_X$ sur un espace analytique $X$ de dimension pure.
\textit{Fonctions de plusieurs variables complexes, III (S\'em. Fran\c cois Norguet, 1975--1977)}
187--204, Lecture Notes in Math., 670, Springer, Berlin, 1978.

\bibitem{BjAbel} \textsc{J.-E. Björk:}
Residues and $\mathcal{D}$-modules.
\textit{The legacy of Niels Henrik Abel}, 605--651, Springer, Berlin, 2004.



\bibitem{Eis} \textsc{D.\ Eisenbud:} Commutative algebra. With a view toward algebraic geometry.
\textit{Graduate Texts in Mathematics}, \textbf{150}. Springer-Verlag, New-York, 1995.


\bibitem{FOV} \textsc{J.\ E.\ Forn\ae ss, N.\ \O vrelid, S.\ Vassiliadou:} Semiglobal
results for $\debar$ on a complex space with arbitrary singularities.
\textit{Proc. Am. Math. Soc.}, \textbf{133(8)} (2005), 2377--2386.

\bibitem{HePa} \textsc{G.\ Henkin, M.\ Passare:} Abelian differentials on singular varieties and variations on a theorem
of Lie-Griffiths.
\textit{Invent. Math.}, \textbf{135(2)} (1999), 297--328.


\bibitem{HePo2} \textsc{G.\ Henkin, P.\ Polyakov:}
Residual $\overline\partial$-cohomology and the complex Radon transform on subvarieties of $\Bbb{C}P^n$.
\textit{Math.\ Ann.}, \textbf{354} (2012), 497--527.

\bibitem{HePo3} \textsc{G.\ Henkin, P.\ Polyakov:}
Explicit Hodge-type decomposition on projective complete intersections.
\textit{J. Geom. Anal.}, \textbf{26} (2016), no. 1, 672--713.

\bibitem{HeLi} \textsc{M. Herrera, D. Lieberman:}
Residues and principal values on complex spaces.
\textit{Math. Ann.}, \textbf{194} (1971), 259--294.

\bibitem{Ho1} \textsc{L. H\"ormander:}
The analysis of linear partial differential operators. I. Distribution theory and Fourier analysis.
\textit{Grundlehren der Mathematischen Wissenschaften}, 256. Springer-Verlag, Berlin, 1983.

\bibitem{LarComp} \textsc{R.\ L\"ark\"ang:}
A comparison formula for residue currents.
\textit{Math. Scand} (to appear), available at arXiv:1207.1279 [math.CV].

\bibitem{LExpl} \textsc{R.\ L\"ark\"ang:}
Explicit versions of the local duality theorem in $\mathbb{C}^n$.
\textit{Preprint}, 2015. Available at arXiv:1510.01965 [math.CV].

\bibitem{LR2} \textsc{R.\ L\"ark\"ang, J.\ Ruppenthal:}
Koppelman formulas on affine cones over smooth projective complete intersections.
\textit{Indiana Univ. Math. J.}, \textbf{67} (2018), no. 2, 753--780.

\bibitem{Mal} \textsc{B.\ Malgrange:} Sur les fonctions diff\'{e}rentiables et les ensembles analytiques.
\textit{Bull. Soc. Math. France}, \textbf{91} (1963), 113--127.

\bibitem{MalBook} \textsc{B.\ Malgrange:} Ideals of differentiable functions.
\textit{Tata Institute of Fundamental Research Studies in Mathematics}, No. 3,
Tata Institute of Fundamental Research, Bombay, 1967.


\bibitem{Nar} \textsc{R. Narasimhan:} Introduction to the theory of analytic spaces.
\textit{Lecture Notes in Mathematics}, No. 25, Springer-Verlag, Berlin, 1966.



\bibitem{OvVass2} \textsc{N.\ \O vrelid, S.\ Vassiliadou:}
$L^2$-$\dbar$-cohomology groups of some singular complex spaces.
\textit{Invent. Math.}, \textbf{192} (2013), no. 2, 413--458.

\bibitem{PaSt1} \textsc{W.\ Pardon, M.\ Stern:} $L^2$-$\debar$-cohomology of complex projective varieties.
\textit{J. Amer. Math. Soc.}, \textbf{4(3)} (1991), 603--621.

\bibitem{PaSt2} \textsc{W.\ Pardon, M.\ Stern:} Pure Hodge structure on the $L^2$-cohomology of
varieties with isolated singularities.
\textit{J. Reine Angew. Math.}, \textbf{533} (2001), 55--80.





\bibitem{Roos}\textsc{J-E.\ Roos: } Bidualit\'e et structure des foncteurs d\'eriv\'es de
$\varinjlim$ dans la cat\'egorie des modules sur un anneau r\'egulier.
\textit{C.\  R.\  Acad.\  Sci.\  Paris}, \textbf{\bf 254}   (1962),  1720--1722.







\bibitem{Rupp2} \textsc{J.\ Ruppenthal:} $L^2$-theory for the $\debar$-operator on
compact complex spaces.
\textit{Duke Math. J.}, \textbf{163} (2014), no. 15, 2887--2934.


\bibitem{Sz}\textsc{J.\ Sznajdman:} A Brian\c con-Skoda type result for a non-reduced analytic space.
\textit{J. Reine Angew. Math.} (to appear)  available at \textit{arXiv:1001.0322 [math.CV]}.


\bibitem{Ts} \textsc{A. K. Tsikh:}
Multidimensional residues and their applications.
\textit{Translations of Mathematical Monographs}, 103. American Mathematical Society, Providence, RI, 1992.


\end{thebibliography}
\end{document}